\documentclass[11pt,reqno]{amsart}
\usepackage{amsaddr}
\usepackage[pdftex]{graphicx}
\usepackage{pgfplots}
\usepackage{orcidlink}
\usepackage{geometry}
\usepackage{stix}
\usepackage{subcaption}

\usepackage{caption}
\usepackage{amsmath,amsthm}
\usepackage{cite}
\usepackage{hyperref}
\hypersetup{
    colorlinks=true, %set true if you want colored links
    linktoc=all,     %set to all if you want both sections and subsections linked
    allcolors=blue,  %choose some color if you want links to stand out
}
\usepackage{url}
\newtheorem{theorem}{Theorem}[section]
\newtheorem{lemma}{Lemma}[section]
\newtheorem{proposition}{Proposition}[section]
\newtheorem{corollary}{Corollary}[section]
\newtheorem{conjecture}{Conjecture}[section]
\theoremstyle{definition}
\newtheorem{definition}{Definition}[section]

\newtheorem{remark}{Remark}[section]
\numberwithin{equation}{section}
\pgfplotsset{compat=1.18}
\definecolor{wewdxt}{rgb}{0.45,0.45,0.45}
\definecolor{uuuuuu}{rgb}{0.3,0.3,0.3}
\definecolor{zzttqq}{rgb}{0.6,0.2,0.}
\definecolor{qqwwtt}{rgb}{0.,0.4,0.2}
\definecolor{xfqqff}{rgb}{0.5,0.,1.}
\definecolor{yqyqyq}{rgb}{0.5,0.5,0.5}
\definecolor{xdxdff}{rgb}{0.5,0.5,1.}
\definecolor{ududff}{rgb}{0.3,0.3,1.}

\begin{document}
\pagenumbering{arabic}

\title[Generalized Chapple--Euler Relation]{Generalized Chapple--Euler Relation}
\author{Vladimir Dragovi\'{c}$^{1,2}$ and Mohammad Hassan Murad$^{3}$\orcidlink{0000-0002-8293-5242}}
\address{$^{1}$Department of Mathematical Sciences
\\The University of Texas at Dallas, Richardson, TX, USA\\
$^{2}$Mathematical Institute SANU, Belgrade, Serbia\\
$^{3}$Department of Mathematics\\
The University of Texas at Arlington, Arlington, TX, USA
}
\email{vladimir.dragovic@utdallas.edu; mohammad.murad2@uta.edu}

\begin{abstract}
We provide a new proof of the necessary and sufficient condition for a triangle to be circumscribed about a central conic (ellipse or hyperbola), expressed in terms of the circumradius and the distances from the circumcenter to the foci. If the inscribed conic is an ellipse, in the limiting case, where the foci coincide, the condition reduces to the classical Chapple--Euler relation.

We also prove that the sum of the squares of the sides of a triangle in a family inscribed in a circle and circumscribed about a central conic remains invariant throughout the family if and only if the center of the circle coincides either with the center of the conic or with one of its foci. 

Using Blaschke products of degree three and affine transformations, we characterize all central $3$-Poncelet pairs with constant triangle area. We prove that the associated family of triangles has constant area if and only if the two conics are homothetic ellipses. We also prove two observations of Reznik concerning the invariance of the total area of the power circles of Poncelet triangles by relating them to Apollonius's identity on medians and establish a more general result valid for both ellipses and hyperbolas. Finally, we propose two conjectures on area invariance of power-circles.
\end{abstract}


%\tableofcontents

\maketitle{}

\section{Introduction}\label{sec:intro}\noindent
For an arbitrary triangle, the relation among its circumradius $R$, inradius $r$, and the distance $d$ between the circumcenter and the incenter is given by the classical formula
\begin{equation}\label{eq:chapple}
    d=\sqrt{R^2-2Rr}.
\end{equation}
This formula is commonly known as the \emph{Euler relation for a triangle}. It was first proved by the English surveyor and mathematician William Chapple (1718--1781) in 1746 \cite{Chapple1746}.

\begin{figure}[htbp]
  \centering
  % First subfigure
  \begin{subfigure}[b]{0.45\textwidth}
    \centering
\begin{tikzpicture}[scale=2.2]
\clip(-1.6,-1.2) rectangle (0.8,1.2);
\draw [line width=1.pt,gray] (0.,0.) circle (0.4cm);
\draw [line width=1.pt,gray] (-0.447213595499958,0.) circle (1.cm);
\draw [line width=1.pt,color=red] (0.524578114876515,0.23584077604089057)-- (0.09600477224665926,-0.8395914512086823);
\draw [line width=1.pt,color=red] (0.09600477224665926,-0.8395914512086823)-- (-1.044081294398846,0.8023396724649395);
\draw [line width=1.pt,color=red] (-1.044081294398846,0.8023396724649395)-- (0.524578114876515,0.23584077604089057);
\draw [line width=1.pt,color=green] (-1.3846614591736213,0.34812569984659975)-- (0.4613882169248855,0.41766343682238877);
\draw [line width=1.pt,color=green] (0.4613882169248855,0.41766343682238877)-- (0.31441547804129516,-0.6480132362357209);
\draw [line width=1.pt,color=green] (0.31441547804129516,-0.6480132362357209)-- (-1.3846614591736213,0.34812569984659975);
\draw [line width=1.pt,color=blue] (-1.333336986551548,-0.46344938864047586)-- (0.4786620116946136,-0.3778284796069286);
\draw [line width=1.pt,color=blue] (0.4786620116946136,-0.3778284796069286)-- (0.2763999226313168,0.6902053870969708);
\draw [line width=1.pt,color=blue] (0.2763999226313168,0.6902053870969708)-- (-1.333336986551548,-0.46344938864047586);
\draw [line width=1.pt,color=brown] (-0.679989396950351,-0.972530424335983)-- (0.5475620335060153,-0.1020854932679976);
\draw [line width=1.pt,color=brown] (0.5475620335060153,-0.1020854932679976)-- (-0.15377396275332728,0.9559776053514655);
\draw [line width=1.pt,color=brown] (-0.15377396275332728,0.9559776053514655)-- (-0.679989396950351,-0.972530424335983);
\begin{scriptsize}
\draw[color=black] (-0.15,0.2) node {$\mathcal{D}$};
\draw[color=black] (-0.75,1.08) node {$\mathcal{C}$};
\draw [fill=red] (-1.044081294398846,0.8023396724649395) circle (0.8pt);
\draw [fill=red] (0.524578114876515,0.23584077604089057) circle (0.8pt);
\draw [fill=red] (0.09600477224665926,-0.8395914512086823) circle (0.8pt);
%\draw [fill=black] (0.,0.) circle (0.6pt);
%\draw [fill=black] (-0.447213595499958,0.) circle (0.6pt);
\draw [fill=green] (-1.3846614591736213,0.34812569984659975) circle (0.8pt);
\draw [fill=green] (0.4613882169248855,0.41766343682238877) circle (0.8pt);
\draw [fill=green] (0.31441547804129516,-0.6480132362357209) circle (0.8pt);
\draw [fill=blue] (-1.333336986551548,-0.46344938864047586) circle (0.8pt);
\draw [fill=blue] (0.4786620116946136,-0.3778284796069286) circle (0.8pt);
\draw [fill=blue] (0.2763999226313168,0.6902053870969708) circle (0.8pt);
\draw [fill=brown] (-0.679989396950351,-0.972530424335983) circle (0.8pt);
\draw [fill=brown] (-0.15377396275332728,0.9559776053514655) circle (0.8pt);
\draw [fill=brown] (0.5475620335060153,-0.1020854932679976) circle (0.8pt);
\end{scriptsize}
\end{tikzpicture}
    \caption{$\mathcal{D}$ is an incircle.}
    \label{fig:illusporism(A)}
  \end{subfigure}
  \begin{subfigure}[b]{0.45\textwidth}
    \centering
\begin{tikzpicture}[scale=1.5]
%[line cap=round,line join=round,>=triangle 45,x=1.0cm,y=1.0cm]
\clip(-3.3,-1.6) rectangle (1.6,1.6);
\draw [line width=1.pt,gray] (0.,0.) circle (1.5cm);
\draw [line width=1.pt,gray] (-2.,0.) circle (1.cm);
\draw [line width=1.pt,color=red] (-2.795214478408659,0.6063282389343616)-- (-1.7598964303844107,0.9707472770282958);
\draw [line width=1.pt,color=red] (-1.7598964303844107,0.9707472770282958)-- (-1.3627585068995551,-0.7706641807370542);
\draw [line width=1.pt,color=red] (-1.3627585068995551,-0.7706641807370542)-- (-2.795214478408659,0.6063282389343616);
\draw [line width=1.pt,color=green] (-2.9689318227693366,0.2473279661196666)-- (-1.5814582237499366,0.9081975454346052);
\draw [line width=1.pt,color=green] (-1.5814582237499366,0.9081975454346052)-- (-1.4359945305701998,-0.8257710520799761);
\draw [line width=1.pt,color=green] (-1.4359945305701998,-0.8257710520799761)-- (-2.9689318227693366,0.2473279661196666);
\draw [line width=1.pt,color=blue] (-2.9861375152486223,-0.16593010883885254)-- (-1.4555737308234356,0.8388087013321249);
\draw [line width=1.pt,color=blue] (-1.4555737308234356,0.8388087013321249)-- (-1.5521657849262551,-0.8941166119747931);
\draw [line width=1.pt,color=blue] (-1.5521657849262551,-0.8941166119747931)-- (-2.9861375152486223,-0.16593010883885254);
\draw [line width=1.pt,color=brown] (-2.77856916046707,-0.6275588118810876)-- (-1.3590006753513848,0.7675414423990532);
\draw [line width=1.pt,color=brown] (-1.3590006753513848,0.7675414423990532)-- (-1.7744657694305497,-0.9742352440973621);
\draw [line width=1.pt,color=brown] (-1.7744657694305497,-0.9742352440973621)-- (-2.77856916046707,-0.6275588118810876);
\begin{scriptsize}
\draw[color=black] (-0.6198688527855291,1.4701439618593204) node {$\mathcal{D}$};
\draw[color=black] (-2.595044707572535,0.9383658471089728) node {$\mathcal{C}$};
\draw [fill=red] (-2.795214478408659,0.6063282389343616) circle (1.0pt);
\draw [fill=red] (-1.7598964303844107,0.9707472770282958) circle (1.0pt);
\draw [fill=red] (-1.3627585068995551,-0.7706641807370542) circle (1.0pt);
%\draw [fill=black] (-2.,0.) circle (0.8pt);
%\draw [fill=black] (0.,0.) circle (0.8pt);
\draw [fill=green] (-2.9689318227693366,0.2473279661196666) circle (1.0pt);
\draw [fill=green] (-1.5814582237499366,0.9081975454346052) circle (1.0pt);
\draw [fill=green] (-1.4359945305701998,-0.8257710520799761) circle (1.0pt);
\draw [fill=blue] (-2.9861375152486223,-0.16593010883885254) circle (1.0pt);
\draw [fill=blue] (-1.4555737308234356,0.8388087013321249) circle (1.0pt);
\draw [fill=blue] (-1.5521657849262551,-0.8941166119747931) circle (1.0pt);
\draw [fill=brown] (-2.77856916046707,-0.6275588118810876) circle (1.0pt);
\draw [fill=brown] (-1.3590006753513848,0.7675414423990532) circle (1.0pt);
\draw [fill=brown] (-1.7744657694305497,-0.9742352440973621) circle (1.0pt);
\end{scriptsize}
\end{tikzpicture}
    \caption{$\mathcal{D}$ is an excircle.}
    \label{fig:illusporism(B)}
  \end{subfigure}
  \caption{Illustration of a Porism.}
  \label{fig:illusporism}
\end{figure}

A similar relation involving the excenter was established by Landen in 1755 \cite{Landen1755}. If $r$ denotes the exradius, then the formula becomes
\begin{equation}\label{eq:landen}
    d=\sqrt{R^2+2Rr}.
\end{equation}
Figure~\ref{fig:illusporism} illustrates both situations.

The Chapple and Landen formulas may be unified into the following theorem, which gives a necessary and sufficient condition for a triangle to be inscribed in and circumscribed about a pair of circles.

\begin{theorem}[Chapple--Landen]\label{thm:chapplelanden}
A triangle may be inscribed in and circumscribed about a pair of circles of radii $R$ and $r$, respectively, if and only if
\begin{equation}\label{eq:chapplelanden}
    (R^2-d^2)^2=4R^2r^2,
\end{equation}
where $d$ is the distance between the centers of the circles.
\end{theorem}

Polygons admitting both a circumcircle and an incircle are called \emph{bicentric}. Analogues of the Chapple--Euler relation for bicentric $n$-gons were later proved by Fuss and subsequently generalized by other authors for $n\geq 4$.

The Chapple--Landen relation \eqref{eq:chapplelanden} reveals a remarkable phenomenon: whenever there exists one triangle inscribed in and circumscribed about a pair of circles, there exists an infinite family of such triangles. In other words, either a pair of circles do not admit any triangle or there exists an infinite family of triangles simultaneously inscribed in and circumscribed about them. Such problems are known as \emph{porisms} in the literature. Motivated by this phenomenon, Poncelet considered the problem for arbitrary smooth conics and, in 1813--14, discovered one of the central results of projective geometry, now known as the \emph{Poncelet Closure Theorem} (see, e.g. \cite{Dragovic-Radnovic2011}).

In this paper we adopt the following terminology.

\begin{definition}\label{def:3ponspair}
A pair of smooth conics $(\mathcal C,\mathcal D)$ is called a \emph{$3$-Poncelet pair} if there exists a triangle inscribed in $\mathcal C$ and circumscribed about $\mathcal D$. Such a triangle will be called a \emph{Poncelet triangle}.

Figure \ref{fig:illusporism} provides some examples of $3$-Poncelet pairs.
\end{definition}

 The principal contributions of this paper may be summarized as follows.
\begin{itemize}
    \item We derive a generalized Chapple--Euler relation for Poncelet triangles associated with a circle and a central conic using analytic geometry tools. 
    \item We identify several invariants associated with such Poncelet families and provide geometric interpretations of these invariants. 
    \item We extend the theory to ellipse--conic Poncelet pairs and apply the resulting framework to recent observations concerning power circles and area invariants associated with Poncelet triangles.
\end{itemize}
 
Among the results obtained, Theorem~\ref{thm:sqsumindep} provides a geometric characterization of the central Poncelet pairs for which the total area of the squares constructed on the sides of a Poncelet triangle remains invariant throughout the family.

To formalize this notion of an invariant, we adopt the following definition.
\begin{definition}\label{def:invariant}
Let $\mathcal P$ be a family of triangles. A geometric object or a numerical quantity associated with a triangle in $\mathcal P$ will be called an \emph{invariant throughout} $\mathcal P$ if it is independent of the choice of a triangle in the family $\mathcal P$.

Examples include fixed points (such as the classical centers), fixed lines (such as the Euler line), fixed circles (auxiliary circle of the central conic), and constant numerical quantities (such as area).
\end{definition}

Polygons inscribed in a circle and circumscribed about a central conic from a confocal family have recently been studied in \cite{Dragovic-Radnovic2024} using the Poncelet theorem together with the theory of elliptic curves and functions, particularly through the Cayley condition characterizing periodic trajectories.

The aim of Sections~\ref{sec:2} and \ref{sec:3} is to characterize the $3$-Poncelet pairs formed by a circle and a central conic and to derive geometric properties of the associated Poncelet triangles \emph{without invoking the Poncelet Closure Theorem or the theory of elliptic curves and functions}. Instead, we employ a classical tool from analytic geometry---the Joachimsthal symbols, reviewed briefly in Section~\ref{sec:2}. The resulting condition \eqref{eq:genchappleeuler} extends the Chapple--Landen relation \eqref{eq:chapplelanden}, which is recovered as a limiting case when the foci of the inscribed conic coalesce.

Section~\ref{sec:4} is devoted to the problem of determining when the area of the triangles in a Poncelet family remains constant. Using Blaschke products of degree 3, we prove that this occurs only for concentric homothetic ellipses, thereby characterizing all central $3$-Poncelet pairs with invariant triangle area.

In Section~\ref{sec:affinetrans}, we extend the generalized Chapple--Euler relation to the case where the circumconic is an ellipse. Using affine transformations, we obtain the necessary and sufficient condition for a pair consisting of an ellipse and a central conic with parallel axes to form a $3$-Poncelet pair.

Finally, in Section \ref{sec:conj}, we prove two observations of Reznik concerning the invariance of the total area of the power circles associated with Poncelet triangles \cite{Reznik2021,Reznik2024}.

Poncelet pairs consisting of a circle and a parabola have recently been studied by similar methods in \cite{Dragovic-Murad2025a} and, from the perspective of the Poncelet theorem and Cayley conditions, in \cite{Dragovic-Murad2025b}. Properties of Poncelet polygons have also been the subject of several recent investigations; see \cite{Rezniketal.2021,Bialy-Tabachnikov2022,Helmanetal.2022,Helmanetal.2023,Garciaetal.2023,Reznik-Garcia2023,Dragovic-Radnovic2025a,Garciaetal.2026} and the references therein.

% Beyond the generalized Chapple--Euler relation, the paper reveals a close connection between Poncelet geometry and the invariance of quadratic perimeter and area functionals, leading to several classification results.

Throughout this paper, unless stated otherwise, $\mathcal D$ denotes a central conic in the standard form
\begin{equation}\label{eq:conicD}
\frac{x^2}{a}+\frac{y^2}{b}=1,
\end{equation}
where $a,b\in\mathbb R$. For $0<b<a$, the conic is an ellipse, whereas for $b<0$ it is a hyperbola.

\section{Triangles inscribed in a circle and circumscribed about a central conic}\label{sec:2}
\subsection{Joachimsthal Symbols and Tangent Constructions}
We begin by recalling a classical tool from analytic geometry, introduced by F.~Joachimsthal (1818--1861) \cite{Joachimsthal1871}, for determining the pair of tangents from a point to a conic.

Let $S$ be a conic in the plane given by
\begin{equation*}
    S(x,y):=ax^2+2bxy+cy^2+2dx+2ey+f=0.
\end{equation*}
We use the same symbol $S$ to denote the associated $3\times 3$ symmetric matrix
\begin{equation*}
S=\begin{pmatrix}
        a & b & d\\
        b & c & e\\
        d & e & f
    \end{pmatrix},
\end{equation*}
so that $\mathbf{x}^T S \mathbf{x}=S(x,y)$, where
\[
\mathbf{x}=\begin{pmatrix} x & y & 1 \end{pmatrix}^T.
\]

The \emph{Joachimsthal symbols} associated with the conic $S(x,y)=0$ and points $A=(x_A,y_A)$ and $B=(x_B,y_B)$ are defined by
\begin{align*}
    S_A(x,y)&:=ax_A x+b (x y_A+x_A y)+cy_A y+d (x+x_A)+e(y+y_A)+f,\\
    S_{AA}&:=S(x_A,y_A)=S_A(x_A,y_A),\\
    S_{AB}&:=S_A(x_B,y_B)=S_B(x_A,y_A)=S_{BA}.
\end{align*}

The \emph{Joachimsthal equation} for the pair of tangents from a point $A$ to the conic $S=0$ is given by
\begin{equation*}
    S\,S_{AA}=S_A^2.
\end{equation*}
For further details, see \cite{Dragovic-Murad2025a}.

\begin{lemma}\label{lemm:tangent}
Let $S(x,y)=0$ be a conic and $A=(x_A,y_A)$ a point in the plane. Denote by $\mathcal{C}(O)$ the circle of unit radius with center $O$.
\begin{itemize}
 \item[(a)] The slopes of the tangents from $A$ to the conic $S$ are the roots of the quadratic equation
\begin{equation}\label{eq:slopeeq}
    Um^2-2Vm+W=0,
\end{equation}
where
\begin{align*}
    U&=(ac-b^2) x_A^2 + 2(cd-be) x_A+ cf-e^2,\\
    V&=(ac-b^2)x_A y_A+(ae-bd)x_A+(cd-be)y_A+de-bf,\\
    W&=(ac-b^2)y_A^2+2(ae-bd)y_A+af-d^2.
\end{align*}

If $U \ne 0$, then the slopes $m_{AB}$ and $m_{AC}$ of the tangents $AB$ and $AC$ are given by 
\begin{equation}\label{eq:slopevals}
    m_{AB}=\frac{V+\sqrt{-S_{AA} \det S}}{U}, \quad
    m_{AC}=\frac{V-\sqrt{-S_{AA} \det S}}{U}.
\end{equation}

If $U=0$ and $V \neq 0$, then one tangent is vertical, while the slope of the other tangent, say $m_{AC}$, is
    \begin{equation}\label{eq:vertslope}
        m_{AC}=\frac{W}{2V}.
    \end{equation}

\item[(b)] Assume that $A\in \mathcal{C}(O)$, and let $B$ and $C$ be the second intersection points of the tangents from $A$ with $\mathcal{C}(O)$. Then the $x$-coordinate of $B$ is
\begin{equation}\label{eq:xBcoord}
        x_{B}=\frac{(m_{AB}^2-1) x_A - 2m_{AB} (y_A-y_O) + 2x_O}{m_{AB}^2 + 1},
\end{equation}
and the analogous formula holds for $C$ with $m_{AC}$. Moreover, the slope of the line $BC$ is
\begin{equation}\label{eq:mBC}
    m_{BC}=\frac{(x_A - x_O) (U-W) + 2V (y_A - y_O)}{(y_A - y_O) (U - W) - 2V (x_A - x_O)}.
\end{equation}

\item[(c)] If $A\in \mathcal{C}(O)$, then $A$ lies on a common tangent to $\mathcal{C}(O)$ and $S=0$ if and only if either $U=0$ or $f_1(A,O)=0$, where
\begin{equation}\label{eq:commtan}
f_1(A,O):= U(x_A-x_O)^2+2V(x_A-x_O)(y_A-y_O)+W(y_A -y_O)^2.
\end{equation}
\end{itemize}
\end{lemma}
\begin{proof}
\begin{itemize}
\item[(a)] The Joachimsthal equation for tangents from $A$,  
\begin{equation*}
    SS_{AA}=S_A^2,
\end{equation*}
reduces to the quadratic equation \eqref{eq:slopeeq}, yielding the stated formulas.
\item[(b)] Since $A\in \mathcal{C}(O)$, one obtains
\begin{equation}\label{eq:xA+xB}
    x_A+x_B=\frac{2(m_{AB}^2x_A+x_O + m_{AB} (y_A-y_O))}{m_{AB}^2 + 1}.
\end{equation}
which implies \eqref{eq:xBcoord}. The formula for $C$ is analogous. Substituting \eqref{eq:xBcoord} and $x_C$  into
\begin{equation}\label{eq:mBC1}
        m_{BC}=\frac{m_{AB}(x_{B}-x_A)-m_{AC}(x_{C}-x_A)}{x_{B}-x_{C}},
\end{equation}
and simplifying, one obtains
\begin{equation}\label{eq:mBC2}
        m_{BC}=\frac{(x_A - x_O) (1-m_{AB}m_{AC}) + (m_{AB} + m_{AC}) (y_A - y_O)}{(y_A - y_O) (1-m_{AB} m_{AC}) - (m_{AB} + m_{AC}) (x_A - x_O)}.
\end{equation}
\item[(c)] Using Vieta's relations
\begin{equation}\label{eq:vietem}
    m_{AB}+m_{AC}=\frac{2V}{U},\qquad m_{AB}m_{AC}=\frac{W}{U}
\end{equation}
equation  \eqref{eq:mBC2} reduces to \eqref{eq:mBC}.

The case $U = 0$ follows by rewriting \eqref{eq:slopeeq} as 
\begin{equation*}
    0-\frac{2V}{m}+\frac{W}{m^2}=0
\end{equation*}
whose roots are
\begin{equation*}
    \frac{1}{m} = \frac{2V}{W}, \qquad \frac{1}{m}=0.
\end{equation*}
Finally, a tangent from $A$ is common to $\mathcal{C}(O)$ and $S=0$ if and only if its slope coincides with the slope of the tangent to $\mathcal{C}(O)$ at $A$. This yields \eqref{eq:commtan}.
\end{itemize}
\end{proof}

\begin{remark}\label{rem:detD}
For the central conic \eqref{eq:conicD}, the determinant of the associated matrix is
\[
\det \mathcal{D} = -\frac{1}{ab}.
\]
In the elliptic case ($a > b > 0$), we have $\det \mathcal{D} < 0$, and the tangents from a point $A$ to $\mathcal{D}$ are real if and only if $S_{AA} \geq 0$ (cf.\ \eqref{eq:slopevals}). Accordingly, we assume the existence of a point $A$ satisfying $S_{AA} \geq 0$, equivalently, that the circle $\mathcal{C}(O)$ is not contained in the interior of $\mathcal{D}$.

In the hyperbolic case, we analogously assume the existence of a point $A \in \mathcal{C}(O)$ such that $S_{AA} \leq 0$.

Finally, a point $B$ lies on a tangent from $A$ to the conic $S = 0$ if and only if $S_{AB} = 0$.
\end{remark}

\subsection{Criterion for 3-Poncelet pairs} We will use the following result obtained in \cite{Dragovic-Murad2025a}.
\begin{proposition}\label{prop:sBBsCC}
Let $\mathcal C_1, \mathcal{C}_2$ be a pair of nowhere tangential smooth conics. A triangle $\triangle ABC$ may be inscribed in $\mathcal{C}_1$ and circumscribed about $\mathcal{C}_2$ if and only if 
\[
S_{BB}S_{CC}=S^2_{BC}
\]
where $S(x,y)=0$ is the equation of the inconic $\mathcal C_2$.
\end{proposition}

\begin{lemma}\label{lemm:sBBsCC} 
Let $\mathcal{C}$ be the unit circle with center $O$, and let $\mathcal{D}$ be the central conic given by \eqref{eq:conicD}. 
If the tangents from a point $A \in \mathcal{C}$ to $\mathcal{D}$ meet $\mathcal{C}$ again at (not necessarily distinct) points $B$ and $C$, then
\[
S_{BB}S_{CC} - S_{BC}^2
=
-\frac{16\, S_{AA}\, f_1(x_A,y_A)\, f_2(x_O,y_O)}{\bigl(f_3(x_A,y_A)\bigr)^2},
\]
where
\begin{subequations}\label{eq:sBBsCCf}
\begin{align}
f_1(x_A,y_A) 
&:= a(x_A - x_O)^2 + b(y_A - y_O)^2 
- \bigl(1 + x_O(x_A - x_O) + y_O(y_A - y_O)\bigr)^2, \label{eq:sBBsCCfa}\\
f_2(x_O,y_O) 
&:= (x_O^2 + y_O^2 - 1)^2 - 2(a - b)(x_O^2 - y_O^2) 
+ (a - b)^2 - 2(a + b), \label{eq:sBBsCCfb}\\
f_3(x_A,y_A) 
&:= (x_A^2 + y_A^2 - a + b)^2 + 4y_A^2(a - b).
\label{eq:sBBsCCfc}
\end{align}
\end{subequations}
\end{lemma}
\begin{proof}
Take $A\in \mathcal{C}$. Note that $f_3(x_A,y_A) = 0$ if and only if 
\begin{center}
    either $y_A=0$ and $x_A^2=a-b$, or $a=b$ and $x_A=0,\,y_A=0$.
\end{center}
 In either case, 
\begin{equation*}
    S_{AA}=-\frac{b}{a}<0
\end{equation*}
if $b>0$. Thus, for real tangents, in the case of an ellipse, $f_3(x_A,y_A) \neq 0$ (see Remark \ref{rem:detD}).  

A similar conclusion also holds for the case of hyperbola. 

Now, we use Lemma \ref{lemm:tangent} to calculate  $B=(x_B,y_B)$ and $C=(x_C,y_C)$. The rest follows from a computer algebra simplification.
\end{proof}

The following lemma gives the necessary and sufficient condition for a triangle to be inscribed in a circle and circumscribed about a central conic.

\begin{lemma}\label{lemm:3pons} 
Let $\mathcal{C}$ be a circle of unit radius with center $O$, and let $\mathcal{D}$ be a central conic given by \eqref{eq:conicD}. Then $(\mathcal{C},\mathcal{D})$ is a $3$-Poncelet pair if and only if
\begin{equation}\label{eq:3pons}
\left(x_O^2+y_O^2-1\right)^2 - 2(a-b)\left(x_O^2-y_O^2\right) + (a-b)^2 - 2(a+b) = 0.
\end{equation}
\end{lemma}

\begin{proof}
Equation \eqref{eq:3pons} is equivalent to the condition
\[
f_2(O,a,b)=0
\]
(see \eqref{eq:sBBsCCfb}).

Assume first that \(f_2(O,a,b)=0\). Then, by Lemma~\ref{lemm:sBBsCC},
\[
S_{BB}S_{CC}-S_{BC}^2=0.
\]
Since \(f_2(O,a,b)\) is independent of the choice of \(A\), and the equations
\[
S(x,y)=0, \qquad f_1(x_A,y_A)=0
\]
define conics, there are at most four points \(A\in\mathcal C\) for which
\[
S_{AA}\,f_1(x_A,y_A)=0.
\]
Hence we may choose \(A\in\mathcal C\) such that
\[
S_{AA}\,f_1(x_A,y_A)\neq 0.
\]
For this choice, \(A,B,C\) are distinct, and Proposition~\ref{prop:sBBsCC} implies that \(\triangle ABC\) is circumscribed about \(\mathcal D\). Therefore \((\mathcal C,\mathcal D)\) is a \(3\)-Poncelet pair.

Conversely, suppose that \((\mathcal C,\mathcal D)\) is a \(3\)-Poncelet pair, and let \(\triangle ABC\) be a nondegenerate triangle inscribed in \(\mathcal C\) and circumscribed about \(\mathcal D\). By Proposition~\ref{prop:sBBsCC},
\[
S_{BB}S_{CC}=S_{BC}^2.
\]
If \(f_2(x_O,y_O)\neq 0\), then Lemma~\ref{lemm:sBBsCC} yields
\[
S_{AA}=0
\quad\text{or}\quad
f_1(x_A,y_A)=0.
\]
In either case, \(\triangle ABC\) is degenerate, a contradiction. Hence \(f_2(O,a,b)=0\), which is equivalent to \eqref{eq:3pons}.
\end{proof}

Since the condition \eqref{eq:3pons} is independent of the choice of the initial point \(A\in\mathcal C\), it follows that whenever it is satisfied, every admissible choice of \(A\) generates a triangle inscribed in \(\mathcal C\) and circumscribed about \(\mathcal D\). Hence there exist infinitely many such triangles. Conversely, the existence of one nondegenerate such triangle implies \eqref{eq:3pons} by Lemma~\ref{lemm:3pons}. Therefore, \eqref{eq:3pons} is the necessary and sufficient condition for \((\mathcal C,\mathcal D)\) to form a \(3\)-Poncelet pair.

By applying the scaling transformation
\begin{align*}
x &\mapsto \frac{x}{R}, \qquad
y \mapsto \frac{y}{R}, \qquad
x_O \mapsto \frac{x_O}{R}, \qquad
y_O \mapsto \frac{y_O}{R},\\
a &\mapsto \frac{a}{R^2}, \qquad
b \mapsto \frac{b}{R^2},
\end{align*}
one obtains the corresponding condition for a circle of radius \(R\):
\begin{equation}\label{eq:3ponsscaled}
\left(x_O^2+y_O^2-R^2\right)^2
-2(a+b)R^2
+(a-b)^2
-2(a-b)(x_O^2-y_O^2)=0.
\end{equation}

\subsection{The Generalized Chapple--Euler Relation}

Let \(F_{\pm}=\bigl(\pm\sqrt{a-b},0\bigr)\) denote the foci of the central conic \(\mathcal D\), and let \(d_{\pm}\) be their distances from the center \(O\) of the circumcircle. Since
\[
d_{\pm}^{\,2}
=
\left(x_O\mp\sqrt{a-b}\right)^2+y_O^2,
\]
substituting these expressions into \eqref{eq:3ponsscaled} yields the following geometric characterization of circle--conic \(3\)-Poncelet pairs.

\begin{theorem}[Generalized Chapple--Euler Formula I]\label{thm:genchappleeuler}
A triangle with circumradius \(R\) may be circumscribed about a central conic \(\mathcal D\) if and only if
\begin{equation}\label{eq:genchappleeuler}
\left(R^2-d_{+}^2\right)\left(R^2-d_{-}^2\right)
=
4\varepsilon\beta^2R^2,
\end{equation}
where \(\beta\) is the semi-minor axis of \(\mathcal D\), \(d_{+}\) and \(d_{-}\) are the distances from the circumcenter of the triangle to the foci of \(\mathcal D\), and
\[
\varepsilon=
\begin{cases}
1, & \text{if }\mathcal D\text{ is an ellipse},\\
-1, & \text{if }\mathcal D\text{ is a hyperbola}.
\end{cases}
\]
\end{theorem}

\begin{remark}\label{rem:goldbergzwas}
In the case $\varepsilon=1$, formula \eqref{eq:genchappleeuler} was previously obtained by Goldberg and Zwas \cite{Goldberg-Zwas1976} via the characteristic polynomial 
\[
P(\lambda)=\det(\lambda\mathcal{C}+\mathcal{D}),
\]
where $\mathcal{C}$ is the circumcircle of a triangle circumscribed about an ellipse $\mathcal{D}$, under the assumption that $\mathcal{D}$ is entirely contained in $\mathcal{C}$. 

The proof of Theorem~\ref{thm:genchappleeuler} (see Lemma~\ref{lemm:3pons}) makes no assumptions on the position of the foci relative to the circumcircle, nor on whether the central conic is an ellipse or a hyperbola. Therefore, \eqref{eq:genchappleeuler} remains valid for ellipses not necessarily contained in the circumcircle and extends verbatim to hyperbolas upon setting $\varepsilon=-1$.
\end{remark}

\begin{remark}\label{rem:genchappeuler}
In the limiting case when the foci of the ellipse coincide, the conic $\mathcal{D}$ reduces to a circle of radius $\beta=r$, and $d_+=d_-=d$ denotes the distance between the centers. In this case, \eqref{eq:genchappleeuler} reduces to the classical Chapple--Euler formula \eqref{eq:chapplelanden}.
\end{remark}

\begin{corollary}\label{cor:nondenconic}
Let $\mathcal C$ be a circle, and let $F_-,F_+$ be two distinct points in the plane. Then there exists a unique conic $\mathcal D$ with foci $F_+$ and $F_-$ such that $(\mathcal C,\mathcal D)$ forms a $3$-Poncelet pair if and only if neither of the following conditions is satisfied:
\begin{itemize}
    \item[(a)] $F_-$ or $F_+$ (or both) lies on $\mathcal C$;
    \item[(b)] $F_+$ and $F_-$ are inverse points with respect to $\mathcal C$.
\end{itemize}
\end{corollary}

\begin{proof}
We first assume that none of the conditions in (a)-(b) is satisfied. Suppose that $\beta$ is the semi-minor axis length of $\mathcal D$. Now we construct the conic $\mathcal{D}$ analytically. Place the midpoint of $\overline{F_-F_+}$ at the origin, take the line $F_-F_+$ as the $x$-axis, and its perpendicular bisector as the $y$-axis. Let $2c=\operatorname{dist}(F_+,F_-)$.

Set
\begin{align*}
    b &= \varepsilon \beta^2, \\
    a &= b + c^2
\end{align*}
where $\varepsilon \beta^2$ can be obtained from \eqref{eq:genchappleeuler}. Then the conic defined by \eqref{eq:conicD} has foci $F_\pm$. This proves the existence and uniqueness of $\mathcal D$.

The conic $\mathcal{D}$ fails to exist if and only if either $a=0$ or $b=0$.

\noindent
By \eqref{eq:genchappleeuler}, $b=0$ if and only if $R=d_+$ or $R=d_-$. Therefore, $F_+ \in \mathcal C$ or  $F_- \in \mathcal C$. This proves (a). 

For (b), observe that \eqref{eq:3ponsscaled} can be put in the form
\[
\left(x_O^2+y_O^2 - R^2 - a + b\right)^2 + 4(a-b)y_O^2 = 4aR^2.
\]
Since $a \geq b$, the left-hand side is a sum of nonnegative terms, it follows that $a=0$ if and only if either
\[
y_O = 0, \qquad R^2 = x_O^2 - a + b = d_+d_-,
\]
or
\[
a = b, \qquad R = \sqrt{x_O^2 + y_O^2} = d_+ = d_-.
\]
In the latter case, the foci coincide $(c=0)$, and the associated confocal family degenerates to concentric circles with $R=d$. The former case proves (b). See Figure \ref{fig:confocicrit}.
\end{proof}

\subsection{Foci, Conic Type, and Geometric Classification}

It was proved in \cite{Dragovic-Radnovic2024} that, given a unit circle and two points in its open unit disk, there exists a unique ellipse with these points as foci such that the ellipse and the circle form a $3$-Poncelet pair. This naturally raises the question of how the relative position of the foci with respect to the circle determines the type of the central conic.

In the next theorem, we provide a complete characterization of this dependence. We describe how the location of the foci determines whether the conic is an ellipse or a hyperbola.

\begin{theorem}\label{thm:confocicrit}
Given a triangle and two distinct points in the plane, if the triangle is circumscribed about a conic $\mathcal{D}$ with those two points as its foci,  then
\begin{itemize}
    \item[(a)] $\mathcal{D}$ is an ellipse if and only if both foci lie either in the interior or in the exterior of the circumcircle of the triangle;
    \item[(b)] $\mathcal{D}$ is a hyperbola if and only if one focus lies in the interior of the circumcircle of the triangle and the other lies in its exterior.
\end{itemize}
\end{theorem}

\begin{proof}
Without loss of generality, assume that $F_-$ lies in the interior of the circumcircle of $\triangle ABC$, while $F_+$ lies in its exterior. This is equivalent to the conditions $d_- < R$ and $d_+ > R$. It then follows from \eqref{eq:genchappleeuler} that
\[
\left(R^2 - d_-^2\right)\left(R^2 - d_+^2\right) < 0,
\]
and hence $\varepsilon \beta^2 < 0$. Since $\beta^2 > 0$, we must have $\varepsilon = -1$, and therefore $\mathcal{D}$ is a hyperbola.

The remaining cases, in which both foci lie either inside or outside the circumcircle, are treated analogously and yield $\varepsilon = 1$, corresponding to an ellipse. See Figure~\ref{fig:confocicrit}.
\end{proof}

\begin{figure}
    \centering
\begin{subfigure}{0.48\textwidth}
\centering
\begin{tikzpicture}[scale=1.5]
\clip(-3.1,-2) rectangle (2.5,2.5);
\draw [line width=1.pt,gray] (-2.,1.) circle (1.cm);
\draw [rotate around={0.:(0.,0.)},line width=1.pt,red] (0.,0.) ellipse (1.8027756377319946cm and 1.5cm);
\draw [line width=1.pt,color=blue] (-2.6245314745902935,1.780999639722115)-- (-1.2297279830366428,1.6377154693773714);
\draw [line width=1.pt,color=blue] (-1.2297279830366428,1.6377154693773714)-- (-1.9044168857563224,0.004578547412464044);
\draw [line width=1.pt,color=blue] (-1.9044168857563224,0.004578547412464044)-- (-2.6245314745902935,1.780999639722115);
\draw [line width=1.pt,gray] (-2,1)-- (-1,0.);
\draw [line width=1.pt,gray] (-2,1)-- (1,0.);
\begin{scriptsize}
\draw[color=gray] (-3,1.5) node {$\mathcal{C}$};
\draw[color=red] (1,1.5) node {$\mathcal{D}$};
\draw [fill=uuuuuu] (-2.,1.) circle (1.0pt);
\draw[color=black] (-1.947484241301289,1.13909431179668) node {$O$};
\draw [fill=blue] (-2.6245314745902935,1.780999639722115) circle (1.0pt);
\draw[color=blue] (-2.7,1.9194224769077086) node {$A$};
\draw [fill=blue] (-1.2297279830366428,1.6377154693773714) circle (1.0pt);
\draw[color=blue] (-1.1,1.761856212798751) node {$C$};
\draw [fill=blue] (-1.9044168857563224,0.004578547412464044) circle (1.0pt);
\draw[color=blue] (-1.95,-0.15) node {$B$};
\draw [fill=uuuuuu] (1.,0.) circle (1.0pt);
\draw[color=black] (1.0725358207870712,0.14492621682349444) node {$F_+$};
\draw [fill=uuuuuu] (-1.,0.) circle (1.0pt);
\draw[color=black] (-0.9308066800268224,0.14492621682349444) node {$F_-$};
\end{scriptsize}
\end{tikzpicture}
    \caption{$F_\pm$ lie outside the circumcircle of $\triangle ABC$.}
    \label{fig:confocicrit(A)}
\end{subfigure}
\begin{subfigure}{0.48\textwidth}
\centering
\begin{tikzpicture}[scale=1.8]
\clip(-1.4,-0.8) rectangle (0.9,1.6);
\draw [line width=1.pt,gray] (-0.2,0.4) circle (1.cm);
\draw [rotate around={0.:(0.,0.)},line width=1.pt,red] (0.,0.) ellipse (0.5310367218940701cm and 0.2863564212655271cm);
\draw [line width=1.pt,color=blue] (-0.22757571634783036,1.3996197176266103)-- (0.636942977497998,-0.14729009895738604);
\draw [line width=1.pt,color=blue] (0.636942977497998,-0.14729009895738604)-- (-0.6436416486956317,-0.49620426663937595);
\draw [line width=1.pt,color=blue] (-0.6436416486956317,-0.49620426663937595)-- (-0.22757571634783036,1.3996197176266103);
\draw [line width=1.pt,gray] (-0.2,0.4)-- (-0.447213595499958,0.);
\draw [line width=1.pt,gray] (-0.2,0.4)-- (0.447213595499958,0.);
\begin{scriptsize}
\draw[color=gray] (-1.15,1) node {$\mathcal{C}$};
\draw[color=red] (0.1,0.4) node {$\mathcal{D}$};
\draw [fill=uuuuuu] (-0.2,0.4) circle (0.8pt);
\draw[color=black] (-0.17,0.52) node {$O$};
\draw [fill=blue] (-0.22757571634783036,1.3996197176266103) circle (0.9pt);
\draw[color=blue] (-0.23,1.55) node {$A$};
\draw [fill=blue] (0.636942977497998,-0.14729009895738604) circle (0.9pt);
\draw[color=blue] (0.7,-0.25) node {$C$};
\draw [fill=blue] (-0.6436416486956317,-0.49620426663937595) circle (0.9pt);
\draw[color=blue] (-0.68,-0.62) node {$B$};
\draw [fill=uuuuuu] (0.447213595499958,0.) circle (0.8pt);
\draw[color=black] (0.3,-0.04297599837455058) node {$F_+$};
\draw [fill=uuuuuu] (-0.447213595499958,0.) circle (0.8pt);
\draw[color=black] (-0.3,-0.033286061792101855) node {$F_-$};
\end{scriptsize}
\end{tikzpicture}
    \caption{$F_\pm$ lie inside the circumcircle of $\triangle ABC$.}
    \label{fig:confocicrit(B)}
\end{subfigure}
\begin{subfigure}{0.48\textwidth}
\centering
\begin{tikzpicture}[scale=2]
\clip(-2.1,-1) rectangle (2.1,2);
\draw [line width=1.pt,gray] (-0.5,0.5) circle (1.cm);
\draw [samples=100,domain=-0.99:0.99,rotate around={0.:(0.,0.)},xshift=0.cm,yshift=0.cm,line width=1.pt,red] plot ({0.9013878188659973*(1+(\x)^2)/(1-(\x)^2)},{0.4330127018922193*2*(\x)/(1-(\x)^2)});
\draw [samples=100,domain=-0.99:0.99,rotate around={0.:(0.,0.)},xshift=0.cm,yshift=0.cm,line width=1.pt,red] plot ({0.9013878188659973*(-1-(\x)^2)/(1-(\x)^2)},{0.4330127018922193*(-2)*(\x)/(1-(\x)^2)});
\draw [line width=1.pt,color=blue] (-0.2038861389751374,1.4551526481714576)-- (-1.0378296847837545,-0.3430535155999335);
\draw [line width=1.pt,color=blue] (-1.0378296847837545,-0.3430535155999335)-- (0.49978874906871124,0.4794462354878397);
\draw [line width=1.pt,color=blue] (0.49978874906871124,0.4794462354878397)-- (-0.2038861389751374,1.4551526481714576);
\draw [line width=1.pt,gray] (-0.5,0.5)-- (-1,0.);
\draw [line width=1.pt,gray] (-0.5,0.5)-- (1,0.);
\begin{scriptsize}
\draw[color=gray] (-1,1.5) node {$\mathcal{C}$};
\draw[color=red] (1.4,0.7) node {$\mathcal{D}$};
\draw [fill=uuuuuu] (-0.5,0.5) circle (0.8pt);
\draw[color=black] (-0.44685315454930874,0.6363828977347673) node {$O$};
\draw [fill=blue] (-0.2038861389751374,1.4551526481714576) circle (0.8pt);
\draw[color=blue] (-0.18,1.5967867932560333) node {$A$};
\draw [fill=blue] (-1.0378296847837545,-0.3430535155999335) circle (0.8pt);
\draw[color=blue] (-1.1,-0.47) node {$B$};
\draw [fill=blue] (0.49978874906871124,0.4794462354878397) circle (0.8pt);
\draw[color=blue] (0.62,0.5) node {$C$};
\draw [fill=uuuuuu] (1.,0.) circle (0.8pt);
\draw[color=black] (1.15,0.14) node {$F_+$};
\draw [fill=uuuuuu] (-1.,0.) circle (0.8pt);
\draw[color=black] (-1.15,0.14) node {$F_-$};
\end{scriptsize}
\end{tikzpicture}
    \caption{$F_-$ lies inside,  $F_+$ lies outside the circumcircle of $\triangle ABC$.}
    \label{fig:confocicrit(C)}
    \end{subfigure}
        \caption{Illustration of Corollary \ref{cor:nondenconic} and Theorem \ref{thm:confocicrit}.}
    \label{fig:confocicrit}
\end{figure}

\begin{corollary}\label{cor:circentell}
If a triangle is circumscribed about a central conic whose center coincides with the circumcenter of the triangle, then the conic is an ellipse.
\end{corollary}
\begin{proof}
Since the center of the conic is the midpoint of its foci and coincides with the circumcenter of the triangle, the foci are symmetric with respect to the center of the circumcircle. Consequently, either both foci lie inside the circumcircle or neither does. By Theorem~\ref{thm:confocicrit}, this implies that the conic is an ellipse.
\end{proof}

\section{Some Properties of 3-Poncelet Pairs}\label{sec:3}

\begin{proposition}\label{prop:Rhat}
Let $\mathcal D$ be a central conic, and let $\mathcal C$ and $\hat{\mathcal C}$ be concentric circles with common center $O$ and radii $R>0$ and $\hat R>0$, respectively. Then $(\mathcal C,\mathcal D)$ is a $3$-Poncelet pair if and only if $(\hat{\mathcal C},\mathcal D)$ is a $3$-Poncelet pair. This occurs precisely when
\[
R\hat R=d_+d_-,
\]
where $d_+$ and $d_-$ denote the distances from $O$ to the foci of $\mathcal D$.
\end{proposition}
\begin{proof}
It is easy to see that \eqref{eq:genchappleeuler} is biquadratic in $R$ and the product of its solutions is
\[
R^2\hat{R}^2=d_+^2d_-^2.
\]
This proves the claim.
\end{proof}

\begin{remark}\label{rem:oohat}
Proposition \ref{prop:Rhat} immediately implies that both foci of $\mathcal{D}$ lie inside $\mathcal{C}(O,R)$ if and only if they lie outside $\hat{\mathcal{C}}$.

Let $\hat{O}$ denote the reflection of $O$ in the center of $\mathcal{D}$. Then the configuration $(\mathcal{C}(\hat{O},R),\mathcal{D})$ is the reflection of $(\mathcal{C}(O,R),\mathcal{D})$ in this center. Consequently, $(\mathcal{C}(\hat{O},R),\mathcal{D})$ is a $3$-Poncelet pair if and only if $(\mathcal{C}(O,R),\mathcal{D})$ is.
\end{remark}

\begin{corollary}\label{cor:sumdiffsemi}
Let $\mathcal D$ be an ellipse with semi-axes $\alpha$ and $\beta$, where $\alpha>\beta>0$. Let $\mathcal C_+$ and $\mathcal C_-$ be the circles concentric with $\mathcal D$ and having radii $\alpha+\beta$ and $\alpha-\beta$, respectively. Then both $(\mathcal C_+,\mathcal D)$ and $(\mathcal C_-,\mathcal D)$ are $3$-Poncelet pairs.
\end{corollary}

\begin{proof}
Since $\mathcal D$ is an ellipse, we have $a>b>0$. Let
\[
\alpha=\sqrt a,\qquad \beta=\sqrt b,
\]
so that $\alpha$ and $\beta$ are the semi-axes of $\mathcal D$. The foci of $\mathcal D$ are located at distance
\[
c=\sqrt{a-b}
\]
from the common center of the ellipse and the circles. Hence
\[
d_+=d_-=c.
\]

Applying Corollary~\ref{prop:Rhat} to the concentric pair $(\mathcal C_+,\mathcal D)$ gives
\[
R=\sqrt a+\sqrt b=\alpha+\beta.
\]
Therefore $(\mathcal C_+,\mathcal D)$ is a $3$-Poncelet pair.

Furthermore, Corollary~\ref{prop:Rhat} yields
\[
\hat R=\frac{d_+d_-}{R}
      =\frac{a-b}{\sqrt a+\sqrt b}
      =\sqrt a-\sqrt b
      =\alpha-\beta.
\]
Thus the concentric circle of radius $\alpha-\beta$ also forms a $3$-Poncelet pair with $\mathcal D$. Hence $(\mathcal C_-,\mathcal D)$ is a $3$-Poncelet pair as well.

This completes the proof.
\end{proof}

\begin{theorem}\label{thm:orthcirc}
Let $\mathcal{P}$ be a family of triangles inscribed in a circle and circumscribed about a central conic $\mathcal{D}$. Then the orthocenters of the triangles in $\mathcal{P}$ lie on a circle whose center $\hat O$ is the reflection of the circumcenter of the triangle with respect to the center of $\mathcal{D}$ and radius $\hat R$ is $d_+ d_-/R$, where $d_+$ and $d_-$ are the distances from the circumcenter the triangle to the foci of $\mathcal{D}$ and $R$ is the circumradius of the triangle (Figures~\ref{fig:orthcircell}--\ref{fig:orthcirchyp}).
\end{theorem}

\begin{proof} Without loss of generality, we take $F=(0,0)$ is the center of $\mathcal D$. Let $G=(x_G,y_G)$ and $O=(x_O,y_O)$ be the centroid and circumcenter of $\triangle ABC\in \mathcal{P}$, respectively. 

By the property of Euler line, the orthocenter $H=(x_H,y_H)$ satisfies
\[
x_H = 3x_G - 2x_O, \qquad y_H = 3y_G - 2y_O
\]
where
\[
x_G=\frac{1}{3}(x_A+x_B+x_C),\qquad y_G=\frac{1}{3}(y_A+y_B+y_C).
\]
Using Lemma \ref{lemm:tangent}, we calculate the coordinates of $B$ and $C$. Then a direct calculation gives
\[
(x_H + x_O)^2 + (y_H + y_O)^2 = \frac{d_+^2 d_-^2}{R^2}.
\]
Therefore, the orthocenter $H$ lies on the circle 
\[
\Gamma: (x + x_O)^2 + (y + y_O)^2 = \hat{R}^2. 
\]
with center $\hat{O}=2F-O$ and radius $\hat{R}$.
This completes the proof.
\end{proof}

\begin{corollary}\label{cor:ohcoin}
Let $\mathcal P$ be a family of triangles inscribed in a circle and circumscribed about a central conic $\mathcal D$. Then the orthocenter of a triangle in $\mathcal P$ remains invariant throughout the family if and only if the circumcenter of the triangle coincides with a focus of $\mathcal D$. In that case, the common orthocenter is precisely the other focus of $\mathcal D$.

In that case, the Euler line of a triangle in $\mathcal P$ remains invariant throughout the family.
\end{corollary}
\begin{proof}
Let $O$ be the center of the circumcircle, and let $d_+$ and $d_-$ denote the distances from $O$ to the two foci of $\mathcal D$.

By Theorem~\ref{thm:orthcirc}, the orthocenters of the triangles in $\mathcal P$ lie on a circle of radius
\[
\hat R=\frac{d_+d_-}{R},
\]
where $R$ is the circumradius.

The orthocenter is invariant throughout the family if and only if the orthocenter locus degenerates to a single point. This occurs precisely when
\[
\hat R=0.
\]
Since $R>0$, we must have $d_+=0$ or $d_-=0$. Hence the circumcenter $O$ coincides with one of the foci of $\mathcal D$.

Conversely, if $O$ coincides with a focus of $\mathcal D$, then either $d_+=0$ or $d_-=0$, and therefore
\[
\hat R=0.
\]
Thus the orthocenter locus reduces to a single point, so the orthocenter is invariant.

Finally, Theorem~\ref{thm:orthcirc} shows that when $O$ coincides with one focus of $\mathcal D$, the orthocenter locus degenerates to the other focus. Hence the common orthocenter is precisely the remaining focus of $\mathcal D$.
\end{proof}

\begin{corollary}\label{cor:ohindep}
Let $\mathcal P$ be a family of triangles inscribed in a circle $\mathcal C$ and circumscribed about a central conic $\mathcal D$. Then the orthocenters of the triangles in $\mathcal P$ lie on a circle concentric with $\mathcal C$ if and only if $\mathcal C$ and $\mathcal D$ are concentric.
\end{corollary}
\begin{proof}
Let $O$ and $F$ denote the centers of the circle $\mathcal C$ and the conic $\mathcal D$, respectively.

By Theorem~\ref{thm:orthcirc}, the orthocenters of the triangles in $\mathcal P$ lie on a circle with center $\hat{O}$, given by
\[
\hat O = 2F-O.
\]
Since the center of $\mathcal{C}$ is $O$,  
\begin{center}
    $\hat O = O$ if and only if $F=O$. 
\end{center}
Therefore, the orthocenter locus is concentric with the circumcircle $\mathcal C$ if and only if  $\mathcal{C}$ and the conic $\mathcal{D}$ are concentric.

This proves the result.
\end{proof}

\begin{figure}
    \centering
\begin{tikzpicture}[scale=1]
\clip(-5.5,-6.5) rectangle (5,6.5);
\draw [rotate around={0.:(0.,0.)},line width=1.pt,color=red] (0.,0.) ellipse (1.7320508075688772cm and 1.cm);
\draw [line width=1.pt,color=xfqqff] (-1.,-2.) circle (1.6357cm);
\draw [line width=1.pt,gray] (1.,2.) circle (3.9147cm);
\draw [line width=1.pt,color=blue] (-2.6306798312214315,3.4639125838525744)-- (-1.4915767685765848,-1.019424000747405);
\draw [line width=1.pt,color=blue] (-1.4915767685765848,-1.019424000747405)-- (3.5681473339179695,-0.9545719421415388);
\draw [line width=1.pt,color=blue] (3.5681473339179695,-0.9545719421415388)-- (-2.6306798312214315,3.4639125838525744);
\draw [line width=1.pt,dash pattern=on 4pt off 4pt] (-2.6306798312214315,3.4639125838525744)-- (-2.5541083004464036,-2.5101586816703114);
\draw [line width=1.pt,dash pattern=on 4pt off 4pt] (-1.4915767685765848,-1.019424000747405)-- (-2.5541083004464036,-2.5101586816703114);
\draw [line width=1.pt,dash pattern=on 4pt off 4pt] (3.5681473339179695,-0.9545719421415388)-- (-2.5541083004464036,-2.5101586816703114);
\draw [line width=1.pt,gray] (1.,2.) circle (1.6357cm);
%\draw [line width=1.pt,color=qqwwtt] (0.3333333333333333,0.6666666666666666) circle (0.5452333333333333cm);
%\draw [line width=1.pt,color=qqwwtt] (0.3333333333333333,0.6666666666666666) circle (1.3049cm);
%\draw [line width=1.pt,color=zzttqq] (0.,0.) circle (0.81785cm);
%\draw [line width=1.pt,color=zzttqq] (0.,0.) circle (1.95735cm);
\draw [line width=1.pt,color=blue] (1.0703356917083953,3.6341870702192884)-- (-0.5897093604852919,1.6148524449182575);
\draw [line width=1.pt,color=blue] (-0.5897093604852919,1.6148524449182575)-- (1.649697116630174,0.4988643809960281);
\draw [line width=1.pt,color=blue] (1.649697116630174,0.4988643809960281)-- (1.0703356917083953,3.6341870702192884);
\draw [line width=1.pt,dash pattern=on 4pt off 4pt] (1.0703356917083953,3.6341870702192884)-- (0.130323447853277,1.7479038961335736);
\draw [line width=1.pt,dash pattern=on 4pt off 4pt] (-0.5897093604852919,1.6148524449182575)-- (0.130323447853277,1.7479038961335736);
\draw [line width=1.pt,dash pattern=on 4pt off 4pt] (1.649697116630174,0.4988643809960281)-- (0.130323447853277,1.7479038961335736);
\draw [line width=1.pt,color=xfqqff] (-1.,-2.) circle (3.9147cm);
\begin{scriptsize}
\draw [fill=wewdxt] (-1.4142135623730951,0.) circle (1.2pt);
\draw[color=black] (-1.2,0.1) node {$F_-$};
\draw [fill=wewdxt] (1.4142135623730951,0.) circle (1.2pt);
\draw[color=black] (1.15,0.1) node {$F_+$};
\draw[color=red] (-1.3248062053567202,0.9) node {$\mathcal{D}$};
\draw [fill=wewdxt] (1.,2.) circle (1.2pt);
\draw[color=black] (1.15,2.2) node {$O$};
\draw[color=black] (-1.2766370229012276,5.7) node {$\mathcal{C}(O,R)$};
\draw [fill=wewdxt] (-2.6306798312214315,3.4639125838525744) circle (1.2pt);
\draw[color=black] (-2.75,3.7070242845916734) node {$A$};
\draw [fill=wewdxt] (-1.4915767685765848,-1.019424000747405) circle (1.2pt);
\draw[color=black] (-1.37,-0.79) node {$B$};
\draw [fill=wewdxt] (3.5681473339179695,-0.9545719421415388) circle (1.2pt);
\draw[color=black] (3.75,-1.0296119901984235) node {$C$};
\draw [fill=xfqqff] (-2.5541083004464036,-2.5101586816703114) circle (1.2pt);
\draw[color=xfqqff] (-2.75,-2.6) node {$H$};
\draw [fill=wewdxt] (0.,0.) circle (1.2pt);
\draw[color=black] (0.16843845076354716,0.1) node {$F$};
\draw[color=black] (0,3.7) node {$\mathcal{C}(O,\hat{R})$};
%\draw [fill=qqwwtt] (-0.18470308862667897,0.4966388803212622) circle (1.2pt);
%\draw[color=qqwwtt] (-0.42564813285419356,0.4315198776181827) node {$G$};
\draw [fill=wewdxt] (1.0703356917083953,3.6341870702192884) circle (1.2pt);
\draw[color=black] (1.228160464784382,3.899701014413644) node {$A'$};
\draw [fill=wewdxt] (-0.5897093604852919,1.6148524449182575) circle (1.2pt);
\draw[color=black] (-0.8270579866499643,1.587580256550003) node {$B'$};
\draw [fill=wewdxt] (1.649697116630174,0.4988643809960281) circle (1.2pt);
\draw[color=black] (1.9,0.5) node {$C'$};
%\draw [fill=qqwwtt] (0.7101078159510925,1.915967965377858) circle (1.2pt);
%\draw[color=qqwwtt] (0.85,2.15) node {$G'$};
\draw [fill=xfqqff] (0.130323447853277,1.7479038961335736) circle (1.2pt);
\draw[color=xfqqff] (0.5,1.9) node {$H'$};
\draw[color=black] (-4.3,1) node {$\mathcal{C}(\hat{O},R)$};
\draw[color=black] (1,-3) node {$\mathcal{C}(\hat{O},\hat{R})$};
%\draw [fill=zzttqq] (-0.7770546329400184,-0.25504167951810675) circle (1.2pt);
%\draw[color=zzttqq] (-0.95,-0.25) node {$N$};
%\draw [fill=zzttqq] (0.5651617239263883,1.8739519480667859) circle (1.2pt);
%\draw[color=zzttqq] (0.55,2.149554051864082) node {$N'$};
\draw [fill=wewdxt] (-1.,-2.) circle (1.2pt);
\draw[color=black] (-1.2,-1.85) node {$\hat{O}$};
\end{scriptsize}
\end{tikzpicture}
    \caption{Illustration of Theorem \ref{thm:orthcirc} and Corollary \ref{cor:orthhatcirc} when $\mathcal{D}$ is an ellipse.}
    \label{fig:orthcircell}
\end{figure}
\begin{figure}
    \centering
\definecolor{zzttqq}{rgb}{0.6,0.2,0.}
\definecolor{qqwwtt}{rgb}{0.,0.4,0.2}
\definecolor{blue}{rgb}{0.,0.,1.}
\definecolor{red}{rgb}{1.,0.,0.}
\definecolor{wewdxt}{rgb}{0.43137254901960786,0.42745098039215684,0.45098039215686275}
\definecolor{xfqqff}{rgb}{0.4980392156862745,0.,1.}
\begin{tikzpicture}[scale=1.2]
\clip(-5,-5) rectangle (5,6);
\draw [samples=100,domain=-0.99:0.99,rotate around={0.:(0.,0.)},xshift=0.cm,yshift=0.cm,line width=1.pt,color=red] plot ({1.3416407864998738*(1+(\x)^2)/(1-(\x)^2)},{0.4472135954999579*2*(\x)/(1-(\x)^2)});
\draw [samples=100,domain=-0.99:0.99,rotate around={0.:(0.,0.)},xshift=0.cm,yshift=0.cm,line width=1.pt,color=red] plot ({1.3416407864998738*(-1-(\x)^2)/(1-(\x)^2)},{0.4472135954999579*(-2)*(\x)/(1-(\x)^2)});
\draw [line width=1.pt,color=xfqqff] (-1.,-2.) circle (2.2361cm);
\draw [line width=1.pt,gray] (1.,2.) circle (2.8636cm);
\draw [line width=1.pt,color=blue] (-1.4158435483395562,3.5374995642100657)-- (1.8583740349308937,-0.7319222126840393);
\draw [line width=1.pt,color=blue] (1.8583740349308937,-0.7319222126840393)-- (-1.349215229399434,0.3625041172699806);
\draw [line width=1.pt,color=blue] (-1.349215229399434,0.3625041172699806)-- (-1.4158435483395562,3.5374995642100657);
\draw [line width=1.pt,dash pattern=on 4pt off 4pt] (-1.4158435483395562,3.5374995642100657)-- (-2.9065917492260636,-0.8316459818260992);
\draw [line width=1.pt,dash pattern=on 4pt off 4pt] (1.8583740349308937,-0.7319222126840393)-- (-2.9065917492260636,-0.8316459818260992);
\draw [line width=1.pt,dash pattern=on 4pt off 4pt] (-1.349215229399434,0.3625041172699806)-- (-2.9065917492260636,-0.8316459818260992);
\draw [line width=1.pt,gray] (1.,2.) circle (2.2361cm);
%\draw [line width=1.pt,color=qqwwtt] (0.3333333333333333,0.6666666666666666) circle (0.7453666666666666cm);
%\draw [line width=1.pt,color=qqwwtt] (0.3333333333333333,0.6666666666666666) circle (0.9545333333333333cm);
%\draw [line width=1.pt,color=zzttqq] (0.,0.) circle (1.11805cm);
%\draw [line width=1.pt,color=zzttqq] (0.,0.) circle (1.4318cm);
\draw [line width=1.pt,color=blue] (1.7185028288462356,4.117521403655688)-- (1.3214987906074975,-0.2128673113492179);
\draw [line width=1.pt,color=blue] (1.3214987906074975,-0.2128673113492179)-- (-0.8036054902428604,0.6782018136017093);
\draw [line width=1.pt,color=blue] (-0.8036054902428604,0.6782018136017093)-- (1.7185028288462356,4.117521403655688);
\draw [line width=1.pt,dash pattern=on 4pt off 4pt] (1.7185028288462356,4.117521403655688)-- (0.23639612921087247,0.5828559059081798);
\draw [line width=1.pt,dash pattern=on 4pt off 4pt] (1.3214987906074975,-0.2128673113492179)-- (0.23639612921087247,0.5828559059081798);
\draw [line width=1.pt,dash pattern=on 4pt off 4pt] (-0.8036054902428604,0.6782018136017093)-- (0.23639612921087247,0.5828559059081798);
\draw [line width=1.pt,color=xfqqff] (-1.,-2.) circle (2.8636cm);
\begin{scriptsize}
\draw [fill=wewdxt] (-1.4142135623730951,0.) circle (1.2pt);
\draw[color=black] (-1.65,-0.1) node {$F_-$};
\draw [fill=wewdxt] (1.4142135623730951,0.) circle (1.2pt);
\draw[color=black] (1.6501359156662927,0.13) node {$F_+$};
\draw[color=red] (2.2710057216853747,0.9211468005085226) node {$\mathcal{D}$};
\draw [fill=wewdxt] (1.,2.) circle (1.2pt);
\draw[color=black] (1.2,2.1) node {$O$};
\draw[color=black] (-2.428742936535019,-1.4523048630074613) node {$\mathcal{C}(\hat{O},\hat{R})$};
\draw[color=black] (-0.778329528129864,4.8) node {$\mathcal{C}(O,R)$};
\draw [fill=wewdxt] (-1.4158435483395562,3.5374995642100657) circle (1.2pt);
\draw[color=black] (-1.501367783240694,3.7975815980146486) node {$A$};
\draw [fill=wewdxt] (1.8583740349308937,-0.7319222126840393) circle (1.2pt);
\draw[color=black] (1.9409230400043438,-0.8707306143313591) node {$C$};
\draw [fill=wewdxt] (-1.349215229399434,0.3625041172699806) circle (1.2pt);
\draw[color=black] (-1.5956771208638456,0.45) node {$B$};
\draw [fill=xfqqff] (-2.9065917492260636,-0.8316459818260992) circle (1.2pt);
\draw[color=xfqqff] (-3.1,-0.7) node {$H$};
\draw [fill=wewdxt] (0.,0.) circle (1.2pt);
\draw[color=black] (0.11760917929007732,0.15) node {$F$};
\draw[color=black] (0.5,3.8290180438890324) node {$\mathcal{C}(O,\hat{R})$};
%\draw [fill=qqwwtt] (-0.30222824760279293,1.056027156264261) circle (1.2pt);
%\draw[color=qqwwtt] (-0.18,0.95) node {$G$};
\draw [fill=wewdxt] (1.7185028288462356,4.117521403655688) circle (1.2pt);
\draw[color=black] (1.878050148255576,4.37915584669075) node {$A'$};
\draw [fill=wewdxt] (1.3214987906074975,-0.2128673113492179) circle (1.2pt);
\draw[color=black] (1.25,-0.4) node {$C'$};
\draw [fill=wewdxt] (-0.8036054902428604,0.6782018136017093) circle (1.2pt);
\draw[color=black] (-0.7940477510670559,1) node {$B'$};
%\draw [fill=qqwwtt] (0.7454653764036242,1.5276186353027268) circle (1.2pt);
%\draw[color=qqwwtt] (0.887802103212483,1.70705794736812) node {$G'$};
\draw [fill=xfqqff] (0.23639612921087247,0.5828559059081798) circle (1.2pt);
\draw[color=xfqqff] (0.47912830684549224,0.7) node {$H'$};
\draw[color=black] (-3.6,0.2) node {$\mathcal{C}(\hat{O},R)$};
% \draw [fill=zzttqq] (-0.9533423714041894,0.5840407343976216) circle (1.2pt);
%\draw[color=zzttqq] (-1.1555668786224709,0.5124730041415321) node {$N$};
%\draw [fill=zzttqq] (0.6181980646052937,1.2914279529540862) circle (1.2pt);
%\draw[color=zzttqq] (0.7149016509033715,1.1) node {$N'$};
\draw [fill=wewdxt] (-1.,-2.) circle (1.2pt);
\draw[color=black] (-1.1,-1.8) node {$\hat{O}$};
\end{scriptsize}
\end{tikzpicture}
    \caption{Illustration of Theorem \ref{thm:orthcirc} and Corollary \ref{cor:orthhatcirc} when $\mathcal{D}$ is a hyperbola.}
    \label{fig:orthcirchyp}
\end{figure}

\begin{corollary}\label{cor:orthcoinfoc}
Let $\triangle ABC$ be a triangle circumscribed about a central conic $\mathcal{D}$. Then the nine-point circle of $\triangle ABC$ coincides with the auxiliary circle of $\mathcal{D}$ if and only if the circumcenter of $\triangle ABC$ coincides with one of the foci of $\mathcal D$.
\end{corollary}

\begin{proof}
Let $O$ be the circumcenter of $\triangle ABC$, and let $F_+$ and $F_-$ denote the foci of $\mathcal D$.

Assume first that $O$ coincides with one of the foci of $\mathcal D$, say $O=F_+$. Then $d_+=0$, and Theorem~\ref{thm:orthcirc} implies that the orthocenter of $\triangle ABC$ is the other focus:
\[
H=F_-.
\]
The nine-point circle is obtained from the circumcircle by the homothety centered at $H$ with ratio $1/2$. Since $H=F_-$ and $O=F_+$, the image of the circumcircle under this homothety has center equal to the midpoint of $\overline{F_+F_-}$ and radius equal to
\[
\frac{|F_+F_-|}{2}.
\]
But this is precisely the auxiliary circle of $\mathcal D$. Hence the nine-point circle coincides with the auxiliary circle.

Conversely, suppose that the nine-point circle coincides with the auxiliary circle of $\mathcal D$. The center of the nine-point circle is the midpoint of $OH$, whereas the center of the auxiliary circle is the midpoint of $\overline{F_+F_-}$. Therefore,
\[
\operatorname{mid}(O,H)=\operatorname{mid}(F_+,F_-).
\]
Since the two circles are equal, their radii are also equal. The radius of the nine-point circle is $R/2$, while the radius of the auxiliary circle is $|F_+F_-|/2$. Hence
\[
R=|F_+F_-|.
\]
By Theorem~\ref{thm:orthcirc}, the orthocenter locus degenerates precisely when the circumcenter coincides with a focus of $\mathcal D$, in which case the orthocenter is the other focus. It follows that $\{O,H\}=\{F_+,F_-\}$, and therefore $O$ coincides with one of the foci of $\mathcal D$.

This completes the proof.
\end{proof}

\begin{corollary}\label{cor:orthhatcirc} 
Let $\mathcal{C}(O,R)$ denote a circle of radius $R$ centered at $O$ and $\mathcal{D}$ be a central conic. Suppose that $(\mathcal{C}(O,R),\mathcal{D})$ is a $3$-Poncelet pair. Let $d_+$ and $d_-$ denote the distances from $O$ to the foci of $\mathcal{D}$. Let $\mathcal{P}$ and $\hat{\mathcal{P}}$ be the (nonempty)  families of Poncelet triangles associated with the pairs $(\mathcal{C}(O,R),\mathcal{D})$ and $(\mathcal{C}(O,\hat{R}),\mathcal{D})$, respectively, where $\hat{R}=d_+d_-/R$. Let $\triangle \hat{A}\hat{B}\hat{C}\in \hat{\mathcal{P}}$ and $\hat H$ be its orthocenter. Then $\hat{H}\in \mathcal{C}(\hat{O},R)$ where $\hat{O}$ is the reflection of $O$ about the center of $\mathcal{D}$.
\end{corollary}
\begin{proof}
    The proof follows from Theorem \ref{thm:orthcirc}. See Figures \ref{fig:orthcircell}--\ref{fig:orthcirchyp}.
\end{proof}

This following identity is classical (see, e.g., \cite{Johnson2007} p. 174 or  \cite{AltshillerCourt1952}, p. 70). We, however, include a proof for completeness.
\begin{theorem}\label{thm:sqsumlengthP}
Let $\triangle ABC$ be a triangle with centroid $G$. For any point $P$ in the plane,
\begin{equation}\label{eq:sqsumlengthP}
|AP|^2+|BP|^2+|CP|^2=|AG|^2+|BG|^2+|CG|^2+3|PG|^2.
\end{equation}
\end{theorem}
\begin{proof} 

Let $A=(x_A,y_A)$, $B=(x_B,y_B)$, $C=(x_C,y_C)$, $P=(x_P,y_P)$, and let $G=(x_G,y_G)$ be the centroid of $\triangle ABC$. 

We write
\begin{align*}
|AP|^2 
&= (x_A-x_P)^2+(y_A-y_P)^2 \\
&= \bigl((x_A-x_G)+(x_G-x_P)\bigr)^2 
  + \bigl((y_A-y_G)+(y_G-y_P)\bigr)^2 \\
&= |AG|^2 + |PG|^2 
+ 2(x_A-x_G)(x_G-x_P) 
+ 2(y_A-y_G)(y_G-y_P).
\end{align*}

Summing analogous expressions for $B$ and $C$, we obtain
\begin{align*}
|AP|^2+|BP|^2+|CP|^2
&= |AG|^2+|BG|^2+|CG|^2 + 3|PG|^2 \\
&\quad + 2\bigl[(x_A+x_B+x_C-3x_G)(x_G-x_P)\bigr] \\
&\quad + 2\bigl[(y_A+y_B+y_C-3y_G)(y_G-y_P)\bigr].
\end{align*}

Since $G$ is the centroid, we have
\[
x_A+x_B+x_C = 3x_G,
\quad
y_A+y_B+y_C = 3y_G,
\]
so the last two terms vanish. This gives the equation \eqref{eq:sqsumlengthP} and proves the lemma.
\end{proof}

\begin{corollary}\label{cor:sqsumlength}
Let $\triangle ABC$ be a triangle with circumcenter $O$ and orthocenter $H$, and let $R$ be its circumradius. Then
\begin{equation}\label{eq:sqsumlength}
    |AB|^2+|BC|^2+|CA|^2=9R^2-|OH|^2.
\end{equation}
\end{corollary}

\begin{proof}
We begin with Theorem \ref{thm:sqsumlengthP}

\noindent\textbf{Step 1.} Taking $P=A$ in \eqref{eq:sqsumlengthP}, we obtain
\[
|AB|^2+|AC|^2 = 4|AG|^2+|BG|^2+|CG|^2.
\]
Similarly, taking $P=B$ and $P=C$, we get analogous expressions. Summing the three resulting identities yields
\begin{equation}\label{eq:sumAG}
|AB|^2+|BC|^2+|CA|^2 = 3\bigl(|AG|^2+|BG|^2+|CG|^2\bigr).
\end{equation}

\medskip

\noindent\textbf{Step 2.} Now take $P=O$ in \eqref{eq:sqsumlengthP}. Since $|OA|=|OB|=|OC|=R$, we obtain
\[
3R^2 = |AG|^2+|BG|^2+|CG|^2+3|OG|^2,
\]
and hence
\begin{equation}\label{eq:AGrelation}
|AG|^2+|BG|^2+|CG|^2 = 3\bigl(R^2 - |OG|^2\bigr).
\end{equation}

\medskip

\noindent\textbf{Step 3.} Substituting \eqref{eq:AGrelation} into \eqref{eq:sumAG}, we obtain
\[
|AB|^2+|BC|^2+|CA|^2
=
9\bigl(R^2 - |OG|^2\bigr).
\]
Finally, using the property of Euler line, $|OH| = 3|OG|$, we obtain the desired relation \eqref{eq:sqsumlength}.
\end{proof}

As an immediate consequence of Corollary \ref{cor:sqsumlength}, for a family $\mathcal P$ of triangles inscribed in a fixed circle and circumscribed about a central conic $\mathcal D$, the sum of the areas of the squares constructed on the sides of a triangle in $\mathcal P$ is invariant throughout the family if and only if the orthocenters of the triangles either coincide or lie on a circle centered at $O$.

\begin{theorem}\label{thm:sqsumconcen}
Let $\triangle ABC$ be a triangle circumscribed about a central conic $\mathcal{D}$. If the circumcenter of $\triangle ABC$ coincides with the center of $\mathcal D$, then
\begin{equation}\label{eq:sqsumconcen}
|AB|^2 + |BC|^2 + |CA|^2 = 9R^2 - \frac{c^4}{R^2},
\end{equation}
where $R$ is the circumradius of $\triangle ABC$ and $2c$ is the distance between the foci of $\mathcal{D}$.
\end{theorem}

\begin{proof}
Let $O$ be the circumcenter and $H$ be the orthocenter of $\triangle ABC$. Without loss of generality, assume $O=(0,0)$. Then $d_+=d_-=c$. 

It then follows from Theorem \ref{thm:orthcirc} that the distance from $O$ to $H$ is
\[
|OH| = \frac{c^2}{R}.
\]

Equation \eqref{eq:sqsumconcen} now follows from Corollary \ref{cor:sqsumlength}.
\end{proof}

\input{Figures/fig.sqsumconcen}

\begin{theorem}\label{thm:sqsumfoccen} 
Let $\triangle ABC$ be a triangle circumscribed about a central conic $\mathcal{D}$. If the circumcenter of $\triangle ABC$ coincides with one of the foci of $\mathcal D$, then
\begin{equation}\label{eq:sqsumfoccen}
|AB|^2 + |BC|^2 + |CA|^2 = 9R^2 - 4c^2,
\end{equation}
where $R$ is the circumradius of $\triangle ABC$ and $2c$ is the distance between the foci of $\mathcal{D}$.
\end{theorem}

\begin{proof}
Let $O$ be the circumcenter of $\triangle ABC$ and let $H$ be its orthocenter. Let $F_\pm$ denote the foci of $\mathcal{D}$. Without loss of generality, assume that $O = F_+$. 

By Corollary~\ref{cor:orthcoinfoc}, we then have $H = F_-$. Hence,
\[
|OH| = |F_+F_-| = 2c.
\]
The desired invariance of $|AB|^2 + |BC|^2 + |CA|^2$ now follows from Corollary \ref{cor:sqsumlength}. Substituting $|OH| = 2c$ into \eqref{eq:sqsumlength} immediately yields \eqref{eq:sqsumfoccen}. See Figure~\ref{fig:sqsumfoccen}.
\end{proof}

\begin{figure}
    \centering
\begin{subfigure}{0.48\textwidth}
\begin{tikzpicture}[scale=0.7]
\clip(-3,-4.1) rectangle (6,4.1);
\draw [rotate around={0.:(0.,0.)},line width=1.pt,color=red] (0.,0.) ellipse (2.cm and 1.4142135623730951cm);
\draw [line width=1.pt,gray] (1.4142135623730951,0.) circle (4.cm);
\draw [line width=1.pt,color=blue] (-1.0935589229147056,3.1162601242568706)-- (-2.4073430601967365,-1.1813995854462884);
\draw [line width=1.pt,color=blue] (-2.4073430601967365,-1.1813995854462884)-- (4.915115545484537,-1.934860538810581);
\draw [line width=1.pt,color=blue] (4.915115545484537,-1.934860538810581)-- (-1.0935589229147056,3.1162601242568706);
\draw [line width=1.pt,dash pattern=on 3pt off 3pt] (-1.5363, -1.2805)-- (-1.0935589229147056,3.1162601242568706);
\draw [line width=1.pt,dash pattern=on 3pt off 3pt] (0.22477, 1.9873)-- (-2.4073430601967365,-1.1813995854462884);
\draw [line width=1.pt,dash pattern=on 3pt off 3pt] (-1.9925, 0.17283)-- (4.915115545484537,-1.934860538810581);
\draw [line width=1.pt,domain=-4.672287221642436:10.139274188933209] plot(\x,{(-0.-0.*\x)/2.8284271247461903});
\draw [line width=1.pt,color=brown] (-2.4073430601967365,-1.1813995854462884)-- (1.9324188630343289,2.0416331413245037);
\draw [line width=1.pt,color=brown] (4.915115545484537,-1.934860538810581)-- (1.9324188630343289,2.0416331413245037);
\draw [line width=1.pt,color=brown] (1.9324188630343289,2.0416331413245037)-- (-1.0935589229147056,3.1162601242568706);
\draw [line width=1.pt,color=green] (0.,0.) circle (2.cm);
\begin{scriptsize}
\draw [fill=wewdxt] (-1.4142135623730951,0.) circle (2.0pt);
\draw[color=black] (-1.7,0.29433934181879784) node {$F_-$};
\draw [fill=wewdxt] (1.4142135623730951,0.) circle (2.0pt);
\draw[color=black] (1.4876612154661273,0.3069235061438408) node {$F_+$};
\draw[color=red] (-0.65,1.62) node {$\mathcal{D}$};
\draw[color=black] (-0.5,3.8) node {$\mathcal{C}$};
\draw [fill=wewdxt] (-1.0935589229147056,3.1162601242568706) circle (2.0pt);
\draw[color=black] (-1.2368103609056884,3.3837516836168553) node {$A$};
\draw [fill=wewdxt] (-2.4073430601967365,-1.1813995854462884) circle (2.0pt);
\draw[color=black] (-2.67,-1.3) node {$B$};
\draw [fill=wewdxt] (4.915115545484537,-1.934860538810581) circle (2.0pt);
\draw[color=black] (5.2,-1.93) node {$C$};
\draw [fill=wewdxt] (0.,0.) circle (2.0pt);
\draw[color=black] (0.072,0.27) node {$N$};
% \draw [fill=gray] (1.,0.) circle (2.0pt);
%\draw[color=black] (1.0,-0.35) node {$P_\lambda$};
\draw [fill=gray] (-1.750450991555721,0.9674302694052911) circle (2.0pt);
\draw [fill=gray] (1.2538862426439001,-1.5581300621284346) circle (2.0pt);
\draw [fill=gray] (1.9107783112849157,0.5906997927231448) circle (2.0pt);
% \draw [fill=wewdxt] (0.47140452079103173,0.) circle (2.0pt);
% \draw[color=black] (0.56,0.25) node {$G$};
\draw [fill=gray] (1.9324188630343289,2.0416331413245037) circle (2.0pt);
\draw [fill=gray] (-1.5363, -1.2805) circle (2.0pt);
\draw [fill=gray] (0.22477, 1.9873) circle (2.0pt);
\draw [fill=gray] (-1.9925, 0.17283) circle (2.0pt);
\draw[color=black] (2.1,2.3) node {$P$};
\end{scriptsize}
\end{tikzpicture}
    \caption{$\mathcal{D}$ is an ellipse.}
    \label{fig:sqsumfoccen(A)}
\end{subfigure}
\begin{subfigure}{0.48\textwidth}
\begin{tikzpicture}[scale=1.1]
\clip(-3.5,-2.5) rectangle (3.5,2.5);
\draw [samples=100,domain=-0.99:0.99,rotate around={0.:(0.,0.)},xshift=0.cm,yshift=0.cm,line width=1.pt,color=red] plot ({1.*(1+(\x)^2)/(1-(\x)^2)},{1.*2*(\x)/(1-(\x)^2)});
\draw [samples=100,domain=-0.99:0.99,rotate around={0.:(0.,0.)},xshift=0.cm,yshift=0.cm,line width=1.pt,color=red] plot ({1.*(-1-(\x)^2)/(1-(\x)^2)},{1.*(-2)*(\x)/(1-(\x)^2)});
\draw [line width=1.pt,gray] (1.4142135623730951,0.) circle (2.cm);
\draw [line width=1.pt,color=blue] (0.6271593105450949,1.8386260100084157)-- (1.3663659234672714,-1.9994275689434555);
\draw [line width=1.pt,color=blue] (1.3663659234672714,-1.9994275689434555)-- (-0.5793116716392714,0.16080155893504008);
\draw [line width=1.pt,color=blue] (-0.5793116716392714,0.16080155893504008)-- (0.6271593105450949,1.8386260100084157);
\draw [line width=1.pt,dash pattern=on 2pt off 2pt] (-1.4142135623730951,0.)-- (0.6271593105450949,1.8386260100084157);
\draw [line width=1.pt,dash pattern=on 2pt off 2pt] (-1.4142135623730951,0.)-- (1.3663659234672714,-1.9994275689434555);
\draw [line width=1.pt,dash pattern=on 2pt off 2pt] (-1.4142135623730951,0.)-- (-0.5793116716392714,0.16080155893504008);
\draw [line width=1.pt,domain=-2:6.551143823081877] plot(\x,{(-0.-0.*\x)/2.8284271247461903});
\draw [line width=1.pt,color=brown] (3.2,1.4)-- (1.3663659234672714,-1.9994275689434555);
\draw [line width=1.pt,color=brown] (3.2,1.4)-- (-0.5793116716392714,0.16080155893504008);
\draw [line width=1.pt,color=brown] (3.2,1.4)-- (0.6271593105450949,1.8386260100084157);
\draw [line width=1.pt,color=green] (0.,0.) circle (1.cm);
\begin{scriptsize}
\draw [fill=wewdxt] (-1.4142135623730951,0.) circle (1.3pt);
\draw[color=black] (-1.5,0.20015494778601045) node {$F_-$};
\draw [fill=wewdxt] (1.4142135623730951,0.) circle (1.3pt);
\draw[color=black] (1.48,0.20015494778601045) node {$F_+$};
\draw[color=red] (2.55,2.05) node {$\mathcal{D}$};
\draw[color=black] (0.08383863348606341,1.8023398686237817) node {$\mathcal{C}$};
\draw [fill=wewdxt] (0.6271593105450949,1.8386260100084157) circle (1.3pt);
\draw[color=black] (0.602561688796465,2.019868891818466) node {$A$};
\draw [fill=wewdxt] (1.3663659234672714,-1.9994275689434555) circle (1.3pt);
\draw[color=black] (1.3639132699778607,-2.2) node {$C$};
\draw [fill=wewdxt] (-0.5793116716392714,0.16080155893504008) circle (1.3pt);
\draw[color=black] (-0.75,0.2629037044767848) node {$B$};
\draw [fill=wewdxt] (0.,0.) circle (1.3pt);
\draw[color=black] (0.0,0.18) node {$N$};
% \draw [fill=wewdxt] (0.7866247084227387,0.) circle (1.3pt);
% \draw[color=black] (0.75,-0.23) node {$P_\lambda$};
\draw [fill=wewdxt] (3.2,1.4) circle (1.3pt);
\draw[color=black] (3.35,1.5) node {$P$};
\draw [fill=gray] (0.9967626170061832,-0.08040077946751989) circle (1.3pt);
\draw [fill=gray] (0.39352712591399996,-0.9193130050042078) circle (1.3pt);
\draw [fill=gray] (0.02392381945291172,0.9997137844717279) circle (1.3pt);
% \draw [fill=wewdxt] (0.47140452079103173,0.) circle (1.3pt);
% \draw[color=black] (0.45,0.22) node {$G$};
\end{scriptsize}
\end{tikzpicture}
    \caption{$\mathcal{D}$ is a hyperbola.}
    \label{fig:sqsumfoccen(B)}
    \end{subfigure}
\caption{Illustration of Theorem~\ref{thm:sqsumfoccen} and Corollary~\ref{cor:sqsumindepP}. Shown in green is the pedal curve of the inellipse $\mathcal D$ with respect to its foci. This curve coincides with the locus of the midpoints of the sides of the Poncelet triangles inscribed in $\mathcal C$ and circumscribed about $\mathcal D$.}
    \label{fig:sqsumfoccen}
\end{figure}

Theorems~\ref{thm:sqsumconcen}--\ref{thm:sqsumfoccen} naturally raise the question of whether the hypotheses imposed therein are not only sufficient but also necessary for the invariance of the quantity
\[
|AB|^2 + |BC|^2 + |CA|^2
\]
throughout a $3$-Poncelet family. The following theorem answers this question completely.

\begin{theorem}\label{thm:sqsumindep} 
Let $\mathcal{P}$ be a family of triangles inscribed in a circle $\mathcal C$ circumscribed about a central conic $\mathcal{D}$. Then the sum of the areas of the squares constructed on the sides of a triangle in $\mathcal P$ remains invariant throughout the family if and only if the center of $\mathcal{C}$ coincides either with the center of $\mathcal D$ or with one of its foci. (See Figure \ref{fig:samearea}.) 
\end{theorem}
\begin{proof} The ``if'' direction follows from Theorems \ref{thm:sqsumconcen}--\ref{thm:sqsumfoccen}.

To prove the ``only if'' direction, we assume that the sum $|AB|^2+|BC|^2+|CA|^2$ is independent of the choice of the triangle in the family $\mathcal{P}$. Equivalently, the distance $|OH|$ is independent of the choice of the triangle in the $\mathcal{P}$ (Corollary \ref{cor:sqsumlength}). Now Corollaries \ref{cor:ohcoin}--\ref{cor:ohindep} finish the proof.
\end{proof}

\input{Figures/fig.samearea}

\begin{remark}
The quantity
\[
\frac{\pi}{4}(|AB|^2+|BC|^2+|CA|^2)
\]
is the sum of the areas of the circles each centered at the midpoint and passing through the vertices of a side. Therefore, Theorem \ref{thm:sqsumindep} admits an equivalent geometric interpretation: 
\begin{center}
    {\it the total area of the three circles each centered at the midpoint and passing through the vertices of the corresponding side of a Poncelet triangle remains invariant throughout the associated family if and only if the circumcenter of the triangle coincides either with the center of the conic or with one of its foci}.
\end{center}
\end{remark}

\begin{corollary}\label{cor:sqsumindepP}
Let $\mathcal{P}$ be a family of triangles inscribed in a circle $\mathcal C$ and circumscribed about a central conic $\mathcal{D}$. Suppose that $\triangle ABC \in \mathcal{P}$. Then, for any fixed point $P$ in the plane, the sum of the areas of the circles each centered at $P$ and passing through a vertex of $\triangle ABC$ remains invariant throughout the family if and only if the center of $\mathcal C$ coincides with one of the foci of $\mathcal{D}$.
\end{corollary}
\begin{proof}
It suffices to show that the sum
\[
|AP|^2+|BP|^2+|CP|^2
\]
is independent of the choice of the triangle $\triangle ABC$ in $\mathcal{P}$.

This follows from Corollaries \ref{cor:ohcoin}--\ref{cor:ohindep} and Theorem \ref{thm:sqsumlengthP}.
\end{proof}

% \clearpage

\section{Area of Poncelet Triangles}\label{sec:4}
In this section, we study a natural problem to find the family of triangles inscribed in a circle $\mathcal C$ and circumscribed about a central conic $\mathcal{D}$ such that all the triangles in the family have equal area.  

It follows from Theorem \ref{thm:confocicrit} (see Figure \ref{fig:confocicrit}) that three essentially different configurations may occur: 
\begin{itemize}
    \item $\mathcal D$ may be an ellipse completely contained inside $\mathcal C$. 
    \item $\mathcal D$ may be an ellipse intersecting $\mathcal C$, with both foci lying outside $\mathcal C$. 
    \item  $\mathcal D$ may be a hyperbola intersecting $\mathcal C$, with one focus inside and the other outside $\mathcal C$. 
\end{itemize}

The latter two cases necessarily contain degenerate members in the associated Poncelet family: as the triangle varies continuously, one eventually obtains a degenerate triangle collapsing to a line segment, and hence having zero area. Therefore, the area function cannot be constant in these configurations. Consequently, the only possible case in which the area may remain invariant is when $\mathcal D$ is an ellipse entirely contained inside $\mathcal C$.

To study this situation, we employ the correspondence between degree-three Blaschke products and Poncelet families of triangles. Using this approach, we prove that the area of the associated triangles is constant if and only if both foci of $\mathcal D$ coincide with the center of $\mathcal C$, that is, if and only if $(\mathcal C, \mathcal D)$ is a 3-Poncelet pair of concentric circles. 

Moreover, as will be shown in Section~\ref{sec:affinetrans} via affine transformations that a pair of homothetic ellipses, including the case of concentric circles, is the unique central $3$-Poncelet pair with invariant triangle area.

\subsection{Blaschke Product of Degree 3}\label{sec:Blaschke3}
Our approach is based on interpreting Poncelet triangles via Blaschke products of degree three. A connection between finite Blaschke products and ellipses was systematically developed in \cite{daepp2002ellipses} (see also \cite{DaeppGorkinshaffervoss2018}).

Let $\mathbb{D}$ denote a unit disk in $\mathbb{C}$ and $\mathbb{T}=\partial \mathbb{D}$ denote the unit circle. A \textit{Blaschke product of degree 3} is defined as
\begin{equation}\label{eq:blaschkedef}
    B(z)=z \frac{z-a_1}{1-\overline{a_1}z} \frac{z-a_2}{1-\overline{a_2}z},
\qquad a_1,a_2\in \mathbb{D}.
\end{equation}
For each $\lambda\in \mathbb{T}$, the three zeros, $z_1,z_2,z_3$ of $B(z)=\lambda$ determine a triangle inscribed in $\mathbb{T}$. A well-known result asserts that the sides of these triangles are tangent to a fixed ellipse whose foci are  $a_1$ and $a_2$. Thus, Blaschke products of degree three provide a natural analytic parametrization of Poncelet families of triangles associated with a circumcircle and an inner ellipse. This framework allows us to translate geometric invariance questions into algebraic conditions on the parameters of the Blaschke product.

Using $B(z)=\lambda$ in \eqref{eq:blaschkedef}, one can write
\begin{subequations}\label{eq:symmpoly}
    \begin{align}
        s_1:=z_1+z_2+z_3 &= a_1+a_2+\lambda \overline{a_1}\,\overline{a_2}, \label{eq:s1} \\
        s_2:=z_1z_2+z_2z_3+z_3z_1 &= a_1a_2+\lambda (\overline{a_1}+\overline{a_2}), \label{eq:s2} \\
        s_3:=z_1z_2z_3 &= \lambda. \label{eq:s3}
    \end{align}
\end{subequations}

For our discussion, we may assume that $a_1, a_2 \in \mathbb{R}$ without losing any generality.

\subsection{Invariance of Areas}

For vertices $z_1,z_2,z_3\in\partial\mathbb D$, the area of the triangle
$\triangle z_1z_2z_3$ is given by
\begin{equation}\label{eq:areaformulacomp}
    \operatorname{Area}(\triangle z_1z_2z_3)
    =
    \frac14 |z_1-z_2|\,|z_2-z_3|\,|z_3-z_1|.
\end{equation}
To study the variation of the area along the Poncelet family, we use the identity
\[
(z_1-z_2)^2(z_2-z_3)^2(z_3-z_1)^2
=
s_1^2s_2^2-4s_2^3-4s_1^3s_3-27s_3^2+18s_1s_2s_3,
\]
where $s_1,s_2,s_3$ denote the elementary symmetric polynomials in
$z_1,z_2,z_3$ (see \eqref{eq:symmpoly}). The right-hand side is precisely the discriminant of the cubic polynomial having zeros $z_1,z_2$, and $z_3$.

Substituting the expressions for $s_1,s_2$, and $s_3$ in terms of the foci and the parameter $\lambda$ into \eqref{eq:areaformulacomp}, we obtain
\[
\operatorname{Area}(\triangle z_1z_2z_3)
=
\frac14
\sqrt{
\left|
C_0\lambda^4+C_1\lambda^3+C_2\lambda^2+C_1\lambda+C_0
\right|
},
\]
where
\begin{align*}
    C_0 &= a_1^2 a_2^2 (a_1-a_2)^2,\\
    C_1 &= 2(a_1+a_2)(a_1a_2^2+a_1-2a_2)
          (a_1^2a_2-2a_1+a_2),\\
    C_2 &= a_1^4a_2^4+4a_1^2a_2^4+8a_1^3a_2^3+4a_1^4a_2^2+a_1^4+a_2^4\\
        &\quad -20a_1a_2^3-24a_1^2a_2^2-20a_1^3a_2
        +18a_1^2+18a_2^2+36a_1a_2-27.
\end{align*}

Since $|\lambda|=1$, the area is independent of $\lambda$ if and only if the self-reciprocal polynomial
\[
P(\lambda)
=
C_0\lambda^4+C_1\lambda^3+C_2\lambda^2+C_1\lambda+C_0
\]
has constant modulus on $\mathbb T$. This is possible only when all coefficients except one vanish. Since $P$ is palindromic, the only admissible form is
\[
P(\lambda)=C_2\lambda^2.
\]
Hence
\[
C_0=C_1=0.
\]
From $C_0=0$ we obtain
\[
a_1=0,\qquad a_2=0,\qquad\text{or}\qquad a_1=a_2.
\]
Substituting $a_1=0$ into $C_1=0$ yields
\[
-4a_2^3=0,
\]
and therefore $a_2=0$. By symmetry, $a_2=0$ implies $a_1=0$.

Finally, substituting $a_2=a_1$ into $C_1=0$ gives
\[
4a_1^3(a_1+1)^2(a_1-1)^2=0.
\]
Thus either $a_1=0$ or $a_1=\pm1$. The latter is equivalent to $a_1\in\mathbb T$, and hence, by Corollary~\ref{cor:nondenconic}, the corresponding conic is degenerate. Since we are considering a $3$-Poncelet pair, this case is impossible. Therefore, $a_1=0$.

Consequently,
\[
a_1=a_2=0.
\]
Hence the conic pair must be centrally symmetric. In this case,
\[
C_2=-27,
\]
and therefore
\[
\operatorname{Area}(\triangle z_1z_2z_3)
=
\frac14\sqrt{|-27\lambda^2|}
=
\frac{3\sqrt3}{4},
\]
recovering the classical area of an equilateral triangle inscribed in the unit circle.

We have therefore proved the following theorem.

\begin{theorem}\label{thm:invareapons}
Let $\mathcal P$ be a family of triangles inscribed in a circle $\mathcal C$ and circumscribed about a central conic $\mathcal D$. Then the area of a triangle in $\mathcal P$ remains invariant throughout the family if and only if $\mathcal C$ and $\mathcal D$ are concentric circles.
\end{theorem}

\section{Affine Transformations and 3-Poncelet Pairs}\label{sec:affinetrans}
\subsection{3-Poncelet pairs of an ellipse and a central conic}
In this section, we extend the generalized Chapple--Euler relation to the setting where the outer conic is an ellipse and the inner conic is a central conic. Using affine transformations, the problem is reduced to the circular case treated in Theorem~\ref{thm:genchappleeuler}. This framework allows us to translate geometric properties of the original Poncelet pair into corresponding properties of its affine image.
\begin{theorem}[Generalized Chapple--Euler Formula II]\label{thm:genchappleeuler2}
Let $\mathcal{D}_1$ be an ellipse and let $\mathcal{D}_2$ be a central conic given by 
\begin{equation*}
\mathcal{D}_1:\frac{(x-x_O)^2}{\alpha_1^2}+\frac{(y-y_O)^2}{\beta_1^2}=1,\quad
\mathcal{D}_2:\frac{x^2}{\alpha_2^2}+\varepsilon\frac{y^2}{\beta_2^2}=1.
\end{equation*}
Then $(\mathcal{D}_1,\mathcal{D}_2)$ is a $3$-Poncelet pair if and only if
\begin{equation}\label{eq:genchappleeuler2}
\left(1-d_{+}^2\right)\left(1-d_{-}^2\right)=4\varepsilon\,\frac{\beta_2^2}{\beta_1^2},
\end{equation}
where 
\begin{equation}\label{eq:d_pm}
    d_\pm^2=\left(\frac{x_O}{\alpha_1}\mp \sqrt{\frac{\alpha_2^2}{\alpha_1^2}-\varepsilon\frac{\beta_2^2}{\beta_1^2}}\right)^2+\left(\frac{y_O}{\beta_1}\right)^2,
\end{equation}
and $\varepsilon =1$ if $\mathcal{D}_2$ is an ellipse and  $\varepsilon =-1$ if $\mathcal{D}_2$ is a hyperbola.
\end{theorem}

\begin{proof}
Consider the affine transformation
\begin{equation}\label{eq:affinet}
    T(\mathbf{x})=Q\mathbf{x},\qquad \mathbf{x}\in \mathbb{R}^2
\end{equation}
where
\[
Q=\begin{pmatrix}
    \alpha_1^{-1} & 0\\
    0 & \beta_1^{-1}
\end{pmatrix}.
\]
Under this transformation, the ellipse $\mathcal{D}_1$ is mapped to the circle
\[
T(\mathcal{D}_1):\left(x-\frac{x_O}{\alpha_1}\right)^2+\left(y-\frac{y_O}{\beta_1}\right)^2=1,
\]
while $\mathcal{D}_2$ is mapped to the central conic
\[
T(\mathcal{D}_2):\frac{x^2}{\left(\frac{\alpha_2}{\alpha_1}\right)^2}
+\varepsilon\frac{y^2}{\left(\frac{\beta_2}{\beta_1}\right)^2}=1.
\]
Affine transformations preserve the Poncelet property; hence, $(\mathcal{D}_1,\mathcal{D}_2)$ is a $3$-Poncelet pair if and only if $(T(\mathcal{D}_1),T(\mathcal{D}_2))$ is a $3$-Poncelet pair.

Applying Theorem \ref{thm:genchappleeuler} to $(T(\mathcal{D}_1),T(\mathcal{D}_2))$ yields condition \eqref{eq:genchappleeuler2}. This completes the proof.
\end{proof}

\begin{figure}
  \begin{subfigure}[b]{0.45\textwidth}
    \centering
\begin{tikzpicture}[scale=1]
%[line cap=round,line join=round,>=triangle 45,x=1.0cm,y=1.0cm]
\clip(-3.5,-2) rectangle (4.5,3);
\draw [rotate around={0.:(0.,0.)},line width=1.pt,red] (0.,0.) ellipse (1.5cm and 1.cm);
\draw [rotate around={0.:(0.5,0.5)},line width=1.pt] (0.5,0.5) ellipse (3.6226cm and 2.cm);
\draw [line width=1.pt,color=blue] (0.9268855338195401,2.486065297776121)-- (-2.5396805167517864,-0.5879911329266132);
\draw [line width=1.pt,color=blue] (-2.5396805167517864,-0.5879911329266132)-- (1.8420228245642383,-1.3576980679404604);
\draw [line width=1.pt,color=blue] (1.8420228245642383,-1.3576980679404604)-- (0.9268855338195401,2.486065297776121);
\begin{scriptsize}
\draw [fill=wewdxt] (0.5,0.5) circle (1.3pt);
\draw[color=wewdxt] (0.5806017487549502,0.6929320970431493) node {$O$};
\draw[color=red] (-0.15,1.2) node {$\mathcal D_2$};
\draw[color=black] (-1.36,2.45) node {$\mathcal D_1$};
\draw [fill=wewdxt] (0.9268855338195401,2.486065297776121) circle (1.3pt);
\draw[color=wewdxt] (1.01,2.68) node {$A$};
\draw [fill=wewdxt] (-2.5396805167517864,-0.5879911329266132) circle (1.3pt);
\draw[color=wewdxt] (-2.75,-0.7) node {$B$};
\draw [fill=wewdxt] (1.8420228245642383,-1.3576980679404604) circle (1.3pt);
\draw[color=wewdxt] (1.93,-1.55) node {$C$};
\end{scriptsize}
\end{tikzpicture}
    \caption{$\mathcal{D}_2$ is an ellipse entirely inside $\mathcal{D}_1$.}
    \label{fig:ellell3pons(A)}
\end{subfigure}
  \begin{subfigure}[b]{0.45\textwidth}
    \centering
\begin{tikzpicture}[scale=1.5]
%[line cap=round,line join=round,>=triangle 45,x=1.0cm,y=1.0cm]
\clip(-2.,-2) rectangle (2,3);
\draw [rotate around={0.:(0.,0.)},line width=1.pt,red] (0.,0.) ellipse (1.5cm and 1.cm);
\draw [rotate around={90.:(0.5,0.5)},line width=1.pt] (0.5,0.5) ellipse (2.cm and 0.86204cm);
\draw [line width=1.pt,domain=-1.5:1.5,lightgray] plot(\x,{(-1.867313630332293-0.9208592272147069*\x)/-1.256543406430998});
\draw [line width=1.pt,domain=-1:1.7,lightgray] plot(\x,{(--2.373148522996754-1.5595363698482416*\x)/0.3993596276828101});
\draw [line width=1.pt,color=blue] (0.9608380575908271,2.190222962007156)-- (-0.29570534884017086,1.2693637347924493);
\draw [line width=1.pt,color=blue] (-0.29570534884017086,1.2693637347924493)-- (1.3601976852736373,0.6306865921589144);
\draw [line width=1.pt,color=blue] (1.3601976852736373,0.6306865921589144)-- (0.9608380575908271,2.190222962007156);
\begin{scriptsize}
\draw [fill=wewdxt] (0.5,0.5) circle (0.9pt);
\draw[color=wewdxt] (0.5636012389906441,0.6486736640004113) node {$O$};
\draw[color=red] (-0.85,1.1) node {$\mathcal D_2$};
\draw[color=black] (-0.25,2) node {$\mathcal D_1$};
\draw [fill=wewdxt] (0.9608380575908271,2.190222962007156) circle (0.9pt);
\draw[color=wewdxt] (1.02,2.33) node {$A$};
\draw [fill=wewdxt] (-0.29570534884017086,1.2693637347924493) circle (0.9pt);
\draw[color=wewdxt] (-0.45,1.3) node {$B$};
\draw [fill=wewdxt] (1.3601976852736373,0.6306865921589144) circle (0.9pt);
\draw[color=wewdxt] (1.5,0.6) node {$C$};
\end{scriptsize}
\end{tikzpicture}
    \caption{$\mathcal{D}_2$ is an ellipse partially inside $\mathcal{D}_1$.}
    \label{fig:ellell3pons(B)}
\end{subfigure}
  \begin{subfigure}[b]{0.45\textwidth}
    \centering
\begin{tikzpicture}
%[line cap=round,line join=round,>=triangle 45,x=1.0cm,y=1.0cm]
\clip(-4.5,-2.5) rectangle (4.5,4.5);
\draw [samples=50,domain=-0.99:0.99,rotate around={0.:(0.,0.)},xshift=0.cm,yshift=0.cm,line width=1.pt,color=red] plot ({2.*(1+(\x)^2)/(1-(\x)^2)},{1.*2*(\x)/(1-(\x)^2)});
\draw [samples=50,domain=-0.99:0.99,rotate around={0.:(0.,0.)},xshift=0.cm,yshift=0.cm,line width=1.pt,color=red] plot ({2.*(-1-(\x)^2)/(1-(\x)^2)},{1.*(-2)*(\x)/(1-(\x)^2)});
\draw [rotate around={90.:(1.,1.)},line width=1.pt] (1.,1.) ellipse (2.183993892948885cm and 1.5cm);
\draw [line width=1.pt,domain=1.85:2.25,lightgray] plot(\x,{(--4.684278556359269-2.3435465470198205*\x)/-0.16240713847248633});
\draw [line width=1.pt,domain=-4:3,lightgray] plot(\x,{(-4.684278556359269-3.4732626845677395*\x)/-5.129497939416247});
\draw [line width=1.pt,domain=-4:5,lightgray] plot(\x,{(--2.275917287637032-3.5149627797632084*\x)/6.6513197090741105});
\draw [line width=1.pt,color=blue] (2.163608833737407,2.3782172303531177)-- (-0.47365622568343735,0.5924841179925797);
\draw [line width=1.pt,color=blue] (-0.47365622568343735,0.5924841179925797)-- (1.9510595856818307,-0.6888835204965149);
\draw [line width=1.pt,color=blue] (1.9510595856818307,-0.6888835204965149)-- (2.163608833737407,2.3782172303531177);
\begin{scriptsize}
\draw[color=red] (4.10,1.5) node {$\mathcal{D}_2$};
\draw[color=black] (0,3) node {$\mathcal{D}_1$};
\draw [fill=wewdxt] (2.163608833737407,2.3782172303531177) circle (1.3pt);
\draw[color=wewdxt] (2.3,2.7) node {$A$};
\draw [fill=wewdxt] (-0.47365622568343735,0.5924841179925797) circle (1.3pt);
\draw[color=wewdxt] (-0.77,0.58) node {$B$};
% \draw [fill=wewdxt] (-2.23606797749979,0.) circle (1.0pt);
% \draw[color=wewdxt] (-2.48,-0.01) node {$F_-$};
% \draw [fill=wewdxt] (2.236067977499718,0.) circle (1.0pt);
% \draw[color=wewdxt] (2.5,-0.1) node {$F_+$};
\draw [fill=wewdxt] (6.7082039324992255,0.) circle (1.0pt);
\draw [fill=wewdxt] (1.,1.) circle (1.3pt);
\draw[color=wewdxt] (1.0866773465810509,1.212820942275401) node {$O$};
\draw [fill=wewdxt] (1.9510595856818307,-0.6888835204965149) circle (1.3pt);
\draw[color=wewdxt] (2.1,-0.95) node {$C$};
\end{scriptsize}
\end{tikzpicture}
    \caption{$\mathcal{D}_2$ is a hyperbola.}
    \label{fig:ellell3pons(C)}
    \end{subfigure}
    \caption{3-Poncelet pairs $(\mathcal{D}_1,\mathcal{D}_2)$ of an ellipse $\mathcal{D}_1$ and a central conic $\mathcal{D}_2$.}
    \label{fig:ellell3pons}
\end{figure}

\begin{corollary}\label{cor:3ponsconell}
Let $\mathcal{D}_1$ and $\mathcal{D}_2$ be concentric ellipses with parallel axes and semi-axis lengths $\alpha_1,\beta_1$ and $\alpha_2,\beta_2$, respectively. Then $(\mathcal{D}_1,\mathcal{D}_2)$ forms a $3$-Poncelet pair if and only if one of the following conditions holds:
\begin{align}
    \frac{\alpha_2}{\alpha_1}+\frac{\beta_2}{\beta_1}&=1,\label{eq:alphabeta1}\\
    \frac{\alpha_2}{\alpha_1}-\frac{\beta_2}{\beta_1}&=\pm 1.  \label{eq:alphabeta2}
\end{align}
\end{corollary}

\begin{proof}
Since the conics are concentric, we have $(x_O,y_O)=(0,0)$. Substituting into \eqref{eq:d_pm}, we obtain
\[
d_+^2=d_-^2=\frac{\alpha_2^2}{\alpha_1^2}-\varepsilon\frac{\beta_2^2}{\beta_1^2}.
\]
Then condition \eqref{eq:genchappleeuler2} becomes
\[
\left(1-\frac{\alpha_2^2}{\alpha_1^2}+\varepsilon\frac{\beta_2^2}{\beta_1^2}\right)^2
=4\varepsilon\,\frac{\beta_2^2}{\beta_1^2}.
\]
This identity forces $\varepsilon=1$, so $\mathcal{D}_2$ must be an ellipse. Hence, rearranging, we obtain
\[
\frac{\alpha_2}{\alpha_1}=\left|\frac{\beta_2}{\beta_1}\pm 1\right|,
\]
which is equivalent to \eqref{eq:alphabeta1} and \eqref{eq:alphabeta2}.
\end{proof}

\begin{remark}
The above proof shows that if an ellipse $\mathcal{D}_1$ and a central conic $\mathcal{D}_2$ are concentric and have parallel axes, then they form a $3$-Poncelet pair if and only if $\mathcal{D}_2$ is also an ellipse. In particular, no hyperbola concentric with $\mathcal{D}_1$ can participate in such a Poncelet pair.

This conclusion is consistent with Corollary \ref{cor:circentell}. Thus, in both the circle--conic and ellipse--conic settings, the concentric configuration necessarily excludes hyperbolas.
\end{remark}

\begin{corollary}\label{cor:unique_homothetic}
Let $(\mathcal{D}_1,\mathcal{D}_2)$ be a $3$-Poncelet pair of concentric ellipses with semi-axis lengths $\alpha_1,\beta_1$ and $\alpha_2,\beta_2$, respectively. If $(\mathcal{D}_1,\mathcal{D}_2)$ is affinely equivalent to a pair of concentric circles, then
\[
\frac{\alpha_2}{\alpha_1}
=
\frac{\beta_2}{\beta_1}
=
\frac12.
\]
In particular, among all concentric $3$-Poncelet pairs of ellipses, the homothetic pair with ratio $1:2$ is the unique one that can be transformed by an affine map into a pair of concentric circles.
\end{corollary}

\begin{proof}
An affine transformation maps ellipses to ellipses and preserves ratios of semi-axes along principal directions. Hence, a pair of concentric ellipses can be mapped to concentric circles if and only if they are homothetic, i.e.,
\[
\frac{\alpha_2}{\alpha_1}=\frac{\beta_2}{\beta_1}.
\]
Combining this with \eqref{eq:alphabeta1} yields
$
2\alpha_2=\alpha_1,
$
which gives the result.
\end{proof}

\begin{theorem}\label{thm:affine_invariants_concentric_invariant}
Let $\mathcal P$ be a family of triangles inscribed in and circumscribed about a pair of homothetic ellipses. If $\alpha$ and $\beta$ are the semi-axes of the inellipse, then the following quantities remain invariant throughout the family:
\begin{itemize}
    \item[(a)] the area of a triangle in $\mathcal P$;
    \item[(b)] the sum of the areas of the squares constructed on the sides of a triangle in $\mathcal{P}$. 
\end{itemize}
More precisely, if $\triangle ABC \in \mathcal P$, then
\begin{align}
\operatorname{Area}(\triangle ABC) &= 3\sqrt{3}\,\alpha\beta, \label{eq:area_invariant}\\
|AB|^2+|BC|^2+|CA|^2 &= 18(\alpha^2+\beta^2). \label{eq:sqsum_invariant}
\end{align}
\end{theorem}

\begin{proof}
Without loss of generality, we may assume the circumellipse $\mathcal{D}_1$ and inellipse $\mathcal{D}_2$ as
\[
\mathcal D_1:\ \frac{x^2}{4\alpha^2}+\frac{y^2}{4\beta^2}=1,
\qquad
\mathcal D_2:\ \frac{x^2}{\alpha^2}+\frac{y^2}{\beta^2}=1.
\]
The affine transformation $T$ in \eqref{eq:affinet} with $\alpha_1=\alpha$ and $\beta_1=\beta$ maps the pair $(\mathcal D_1,\mathcal D_2)$ onto the pair of concentric circles $(T(\mathcal D_1),T(\mathcal D_2))$ where
\[
T(\mathcal D_1):\ x^2+y^2=4,
\qquad
T(\mathcal D_2):\ x^2+y^2=1.
\]
The family $\mathcal P$ transforms into the family $T(\mathcal P)$ of triangles inscribed in $T(\mathcal D_1)$ and circumscribed about $T(\mathcal D_2)$. Since $(T(\mathcal D_1),T(\mathcal D_2))$ is a $3$-Poncelet pair of two concentric circles, every triangle in $T(\mathcal P)$ is equilateral.

Let $T(\triangle ABC)$ be the image of $\triangle ABC\in\mathcal P$ under $T$. Since $T(\mathcal D_1)$ has radius $2$,
\[
\operatorname{Area}(T(\triangle ABC))
=3\sqrt3.
\]
Using the scaling property of affine transformation $T$, we obtain
\[
\operatorname{Area}(\triangle ABC)=\det(Q)^{-1}\operatorname{Area}(T(\triangle ABC))
=3\sqrt3\,\alpha\beta,
\]
which proves \eqref{eq:area_invariant}.

To prove \eqref{eq:sqsum_invariant}, denote
\[
\Sigma_X=
\sum_{\mathrm{cyc}}(x_{T(A)}-x_{T(B)})^2,
\qquad
\Sigma_Y=
\sum_{\mathrm{cyc}}(y_{T(A)}-y_{T(B)})^2.
\]
Since $T(\triangle ABC)$ is equilateral, the rotational symmetry implies that
$
\Sigma_X=\Sigma_Y.
$

Moreover, from \eqref{eq:sqsumconcen}, we get
\[
\Sigma_X+\Sigma_Y:=|T(A)T(B)|^2+|T(B)T(C)|^2+|T(C)T(A)|^2
=36.
\]
Therefore,
\[
\Sigma_X=\Sigma_Y=18.
\]

Note that for any vector $\mathbf{v}=(x,y)^T$:
\[
\|T^{-1}(\mathbf{v})\|^2=\alpha^2x^2+\beta^2y^2.
\]
Hence, we get
\[
|AB|^2+|BC|^2+|CA|^2
=
\alpha^2\Sigma_X+\beta^2\Sigma_Y.
\]
Substituting the values $\Sigma_X=\Sigma_Y=18$ into the last equation gives the result \eqref{eq:sqsum_invariant}.
\end{proof}

\subsection{3-Poncelet pairs of a central conic and a parabola}
\begin{theorem}[Generalized Chapple--Euler Formula III]\label{thm:genchappleeuler3}
Let $\mathcal{D}_1$ be an ellipse with semi-axes $\alpha,\beta$ and let $\mathcal{D}_2$ be a parabola with focal parameter $p$. Suppose that one of the principal axes of $\mathcal D_1$ is parallel to the axis of $\mathcal D_2$. Then $(\mathcal{D}_1,\mathcal{D}_2)$ forms a $3$-Poncelet pair if and only if the point $P=(p\alpha^2/ \beta^2,0)$ lies on $\mathcal{D}_1$ (Figure \ref{fig:ellpar3pons(A)}).
\end{theorem}
\begin{proof} Without loss of generality, we may assume that
\begin{equation*}
\mathcal{D}_1:\frac{(x-x_O)^2}{\alpha^2}+\frac{(y-y_O)^2}{\beta^2}=1,\qquad
\mathcal{D}_2:y^2=4px.
\end{equation*}
The affine transformation $T$ in \eqref{eq:affinet} with $\alpha_1=\alpha$ and $\beta_1=\beta$ maps the pair $(\mathcal D_1,\mathcal D_2)$ onto the pair $(T(\mathcal D_1),T(\mathcal D_2))$ given by
\[
T(\mathcal D_1): \left(x-\frac{x_O}{\alpha}\right)^2+\left(y-\frac{y_O}{\beta}\right)^2=1,
\qquad
T(\mathcal D_2):\ y^2=4\frac{p\alpha}{\beta^2}x.
\]
The pair $(T(\mathcal D_1),T(\mathcal D_2))$ forms a 3-Poncelet pair if and only if the focus of the parabola $T(\mathcal{D}_2)$ lies on the circle $T(\mathcal{D}_1)$ (see, for example, \cite{Dragovic-Murad2025a,Dragovic-Murad2025b}). This condition is equivalent to
\begin{equation}\label{eq:pliesell}
    \left(\frac{p\alpha}{\beta^2}-\frac{x_O}{\alpha}\right)^2+\left(0-\frac{y_O}{\beta}\right)^2=1,
\end{equation}
which is precisely the condition that the point $P=(p\alpha^2/ \beta^2,0)$ lies on $\mathcal{D}_1$.

Finally, since affine transformations preserve the incidence and tangency, $(\mathcal{D}_1,\mathcal{D}_2)$ is a $3$-Poncelet pair if and only if $(T(\mathcal{D}_1),T(\mathcal{D}_2))$ is a $3$-Poncelet pair. This completes the proof.
\end{proof}

\begin{corollary}
    Let $\mathcal{D}_1$ be an ellipse with semi-axes $\alpha,\beta$ and let $\mathcal{D}_2$ be a parabola with the focal parameter $p$. Suppose that the axis of $\mathcal{D}_2$ is parallel to one of the principal axes of $\mathcal{D}_1$, and that the vertex of $\mathcal{D}_2$ coincides with the center of $\mathcal{D}_1$. Then $(\mathcal{D}_1,\mathcal{D}_2)$ forms a $3$-Poncelet pair if and only if
\begin{equation}\label{eq:betapalpha}
        \beta^2=|p|\alpha.
\end{equation}
\end{corollary}

\begin{figure}
  \begin{subfigure}[b]{0.45\textwidth}
    \centering
\begin{tikzpicture}[scale=0.8]
\clip(-4,-4) rectangle (5,5);
\draw [samples=100,rotate around={-90.:(0.,0.)},xshift=0.cm,yshift=0.cm,line width=1.pt,color=red,domain=-12.0:12.0)] plot (\x,{(\x)^2/2/2.0});
\draw [rotate around={90.:(0.4444444444444444,0.)},line width=1.pt,dash pattern=on 4pt off 4pt,gray] (0.4444444444444444,0.) ellipse (3.cm and 2.cm);
\draw [rotate around={90.:(-1.2725099812520633,1.5385551258176957)},line width=1.pt] (-1.2725099812520633,1.5385551258176957) ellipse (3.cm and 2.cm);
\draw [line width=1.pt,domain=-4.935900993548415:9.05658077298987,lightgray] plot(\x,{(-8.962941923927904-3.549203721984111*\x)/5.64015131324793});
\draw [line width=1.pt,domain=-4.935900993548415:9.05658077298987,lightgray] plot(\x,{(--13.635626327995974--3.549203721984111*\x)/6.956695746897937});
\draw [line width=1.pt,domain=-4.935900993548415:9.05658077298987,lightgray] plot(\x,{(--0.40528188101291485--3.5492037219844503*\x)/1.1993448046944897});
\draw [line width=1.pt,color=blue] (-3.1148124975551097,0.37094079032296845)-- (-0.5369983570438477,-1.2512120623133915);
\draw [line width=1.pt,color=blue] (-0.5369983570438477,-1.2512120623133915)-- (0.6623464476506419,2.297991659671059);
\draw [line width=1.pt,color=blue] (0.6623464476506419,2.297991659671059)-- (-3.1148124975551097,0.37094079032296845);
\begin{scriptsize}
\draw[color=red] (4.1,3.7) node {$\mathcal{D}_2$};
% \draw[color=black] (-1.5340590961790155,-1.8821772394550935) node {$eq2$};
\draw [fill=wewdxt] (-1.2725099812522718,1.5385551258180623) circle (2.0pt);
\draw[color=wewdxt] (-1.5982447923557963,1.648036050267859) node {$O$};
\draw[color=black] (-3.058469380377567,-0.6947418601846459) node {$\mathcal{D}_1$};
\draw [fill=wewdxt] (-3.1148124975551097,0.37094079032296845) circle (1.5pt);
\draw[color=wewdxt] (-3.5,0.42) node {$A$};
\draw [fill=wewdxt] (-0.5369983570438477,-1.2512120623133915) circle (1.5pt);
\draw[color=wewdxt] (-0.5,-1.5291559104827983) node {$B$};
\draw [fill=wewdxt] (0.6623464476506419,2.297991659671059) circle (1.5pt);
\draw[color=wewdxt] (0.92,2.6108214929195728) node {$C$};
\draw [fill=brown] (0.4444444444444444,0.) circle (1.5pt);
\draw[color=brown] (0.5519760295663709,-0.16520986672620303) node {$P$};
\end{scriptsize}
\end{tikzpicture}
    \caption{$\mathcal{D}_1$ is an ellipse.}
    \label{fig:ellpar3pons(A)}
\end{subfigure}
  \begin{subfigure}[b]{0.45\textwidth}
    \centering
\begin{tikzpicture}[scale=0.45]
\clip(-10,-6) rectangle (6,8);
\draw [samples=100,rotate around={-90.:(0.,0.)},xshift=0.cm,yshift=0.cm,line width=1.pt,color=red,domain=-12.0:12.0)] plot (\x,{(\x)^2/2/2.0});
\draw [samples=50,domain=-0.99:0.99,rotate around={0.:(-1.7777777777777777,0.)},xshift=-1.7777777777777777cm,yshift=0.cm,line width=1.pt,dash pattern=on 4pt off 4pt,gray] plot ({2.*(1+(\x)^2)/(1-(\x)^2)},{1.5*2*(\x)/(1-(\x)^2)});
\draw [samples=50,domain=-0.99:0.99,rotate around={0.:(-1.7777777777777777,0.)},xshift=-1.7777777777777777cm,yshift=0.cm,line width=1.pt,dash pattern=on 4pt off 4pt,gray] plot ({2.*(-1-(\x)^2)/(1-(\x)^2)},{1.5*(-2)*(\x)/(1-(\x)^2)});
\draw [samples=50,domain=-0.99:0.99,rotate around={0.:(-5.650721619706137,2.487433994408433)},xshift=-5.650721619706137cm,yshift=2.487433994408433cm,line width=1.pt] plot ({2.*(1+(\x)^2)/(1-(\x)^2)},{1.5*2*(\x)/(1-(\x)^2)});
\draw [samples=50,domain=-0.99:0.99,rotate around={0.:(-5.650721619706137,2.487433994408433)},xshift=-5.650721619706137cm,yshift=2.487433994408433cm,line width=1.pt] plot ({2.*(-1-(\x)^2)/(1-(\x)^2)},{1.5*(-2)*(\x)/(1-(\x)^2)});
\draw [line width=1.pt,domain=-15.587769298391878:12.276788981308,lightgray] plot(\x,{(-17.921076915982763-6.246224407718338*\x)/10.580126088341729});
\draw [line width=1.pt,domain=-15.587769298391878:12.276788981308,lightgray] plot(\x,{(--129.44783330524896--6.246224407718337*\x)/28.435193263234574});
\draw [line width=1.pt,domain=-15.587769298391878:12.276788981308,lightgray] plot(\x,{(--0.4261198382699114--1.0597272487835592*\x)/-0.6719901813731124});
\draw [line width=1.pt,color=blue] (-7.711020570160531,2.8585375755680005)-- (1.0740934949937126,-2.327959583366778);
\draw [line width=1.pt,color=blue] (1.0740934949937126,-2.327959583366778)-- (-2.886738386465543,3.918264824352965);
\draw [line width=1.pt,color=blue] (-2.886738386465543,3.918264824352965)-- (-7.711020570160531,2.8585375755680005);
\begin{scriptsize}
\draw[color=red] (4.8,3.8) node {$\mathcal D_2$};
\draw [fill=wewdxt] (-5.6507216197061245,2.487433994407896) circle (3.5pt);
\draw[color=wewdxt] (-5.449812988118093,2.9754545122853684) node {$O$};
\draw[color=black] (-1.4776312759055568,6.11762810552812) node {$\mathcal D_1$};
\draw [fill=wewdxt] (-7.711020570160531,2.8585375755680005) circle (2.5pt);
\draw[color=wewdxt] (-7.495190138436488,3.3608154246641964) node {$A$};
\draw [fill=wewdxt] (1.0740934949937126,-2.327959583366778) circle (2.5pt);
\draw[color=wewdxt] (0.8345341983674118,-2.6271002907606698) node {$B$};
\draw [fill=wewdxt] (-2.886738386465543,3.918264824352965) circle (2.5pt);
\draw[color=wewdxt] (-2.9005023369966145,4.516898161800681) node {$C$};
\draw [fill=brown] (-1.7777777777777777,0.) circle (2.5pt);
\draw[color=brown] (-1.566560717223748,0.48543015537601825) node {$P$};
\end{scriptsize}
\end{tikzpicture}
    \caption{$\mathcal{D}_1$ is a hyperbola.}
    \label{fig:ellpar3pons(B)}
\end{subfigure}
    \caption{3-Poncelet pairs $(\mathcal{D}_1,\mathcal{D}_2)$ of a central conic $\mathcal{D}_1$ and a parabola $\mathcal{D}_2$.}
    \label{fig:ellpar3pons}
\end{figure}

\begin{remark}
Theorem~\ref{thm:genchappleeuler3} admits a hyperbolic counterpart (Figure \ref{fig:ellpar3pons(B)}). If $\mathcal{D}_1$ is a hyperbola with semi-axes $\alpha$ and $\beta$, then $(\mathcal{D}_1,\mathcal{D}_2)$ forms a $3$-Poncelet pair if and only if the point
\[
P=\left(-\frac{p\alpha^2}{\beta^2},0\right) \in \mathcal{D}_1.
\]
\end{remark}

\section{Area-Invariance of Power Circles}\label{sec:conj}
\noindent
In this section, we present and prove several generalizations of invariance phenomena concerning the total area of the power circles associated with Poncelet triangles, originally discovered by Dan Reznik through computer algebra system (CAS) experiments; see the YouTube videos \cite{Reznik2021,Reznik2024}. Our results not only establish these observations rigorously, but also provide necessary and sufficient conditions for such invariance and extend them to a broader class of Poncelet pairs. Finally, we conclude the section by formulating two conjectures that extends to Poncelet polygons with an odd number of sides.

We begin by recalling the notion of a power circle.
\begin{definition}
For a triangle, the circles centered at the midpoints of its sides and passing through the opposite vertices are called the \emph{power circles} of the triangle.
\end{definition}

\begin{theorem}\label{thm:geninvarea}
Let $\mathcal P$ be a family of triangles inscribed in a circle and circumscribed about a central conic $\mathcal D$. Then the total area of the power circles of a triangle in $\mathcal P$ remains invariant throughout the family if and only if the circumcenter of the triangle coincides either with the center of the inconic $\mathcal D$ or with one of its foci.
\end{theorem}

\begin{theorem}\label{thm:invareafoc}
Let $\mathcal P$ be a family of triangles inscribed in and circumscribed about a pair of homothetic ellipses. Then the total area of the power circles of a triangle in $\mathcal P$ remains invariant throughout the family.
\end{theorem}

The problem of determining when the total area of the power circles is an invariant over a Poncelet family of triangles can be reduced to a classical metric relation on medians. Indeed, the radius of each power circle is equal to the length of the corresponding median of the triangle. Consequently, the total area of the three power circles is proportional to the sum of the squares of the medians. As a consequence of Apollonius's theorem, the latter is itself proportional to the sum of the squares of the sides (see Theorem \ref{thm:med} below and see, for example, \cite[Art. 106, p. 70]{AltshillerCourt1952}).
\begin{theorem}\label{thm:med} 
The sum of the squares of the medians of a triangle is equal to three-fourths the sum of the squares of the sides. 

Symbolically,
\begin{equation}\label{eq:med}
m_a^2+m_b^2+m_c^2=\frac{3}{4}\left(|AB|^2+|BC|^2+|CA|^2\right)
\end{equation}
where $m_a$, $m_b$, $m_c$ denote the medians of $\triangle ABC$, respectively.
\end{theorem}
Therefore, the invariance of the total area of the power circles is equivalent to the invariance of the sum of the squares of the medians, and hence to the invariance of
\[
|AB|^2+|BC|^2+|CA|^2.
\]
For Poncelet families of triangles inscribed in a circle and circumscribed about a central conic, Theorem~\ref{thm:sqsumindep} shows that area invariance occurs precisely when the circumcenter coincides either with the center of the conic or with one of its foci. As a consequence, Theorem \ref{thm:geninvarea} follows directly from Theorem \ref{thm:sqsumindep}. 

Moreover, the scope of Theorem~\ref{thm:geninvarea} extends beyond the setting originally observed by Reznik. It applies not only to ellipses whose foci lie inside the circumcircle, but also to ellipses whose foci lie outside the circumcircle and to hyperbolas, thereby providing a unified treatment of all central conics. Theorem~\ref{thm:invareafoc} follows similarly from Theorem~\ref{thm:affine_invariants_concentric_invariant}.

\medskip

We conclude this section by proposing two conjectures that extend Theorems \ref{thm:geninvarea}--\ref{thm:invareafoc} to $p$-gons.

\begin{conjecture}\label{conj:geninvarea}
Let $p>3$ be an odd integer and let $\mathcal P$ be a family of $p$-gons inscribed in a circle $\mathcal C$ and circumscribed about a central conic $\mathcal D$. Then the total area of the power circles of a polygon in $\mathcal P$ is invariant throughout the family if and only if the center of $\mathcal C$ coincides either with the center of the inellipse $\mathcal D$ or with one of its foci.
\end{conjecture}

\begin{conjecture}\label{conj:invareafoc}
Let $p > 3$ be an odd integer and let $\mathcal P$ be a family of $p$-gons inscribed in and circumscribed about a pair of homothetic ellipses. Then the total area of the power circles of a $p$-gon in $\mathcal P$ is invariant throughout the family.
\end{conjecture}

 \vspace{1.5cm}
\noindent
\thanks{\textbf{Acknowledgments}. The authors acknowledge the use of  \textit{Mathematica} and \textit{GeoGebra} for symbolic and numerical computations, as well as for generating the figures.

The authors express their sincere gratitude to the anonymous reviewer for  valuable suggestions.  In particular, they are deeply grateful for an insightful suggestion to employ Blaschke products in the study of the area invariance problem; this idea has been instrumental in transforming an initial conjecture into Theorem~\ref{thm:invareapons}.

They are grateful to Dan Reznik for helpful discussions and for valuable correspondence with the second author.

The first author’s research was supported by the Serbian Ministry of Science, Technological Development and Innovation, Science Fund of Serbia (grant IntegraRS), and Simons Foundation (grant no.~854861).

\bibliographystyle{amsplain}  % or amsalpha
\bibliography{references}
\end{document}